\documentclass{article}
\usepackage[utf8]{inputenc}
\usepackage[T1]{fontenc}
\usepackage[utf8]{inputenc}
\usepackage{lmodern}
\usepackage{amsmath,amssymb}
\usepackage{geometry}
\usepackage{amsthm}
\usepackage[british]{babel}
\def\1{\mathbbm{1}}
\usepackage{graphicx}
\usepackage{dsfont}
\usepackage{enumerate}
\usepackage{bbm}
\usepackage{bigints}
\usepackage{color}
\usepackage{tikz}
\usepackage{comment}

\newtheorem{theorem}{Theorem}

\newtheorem{proposition}{Proposition}
\newtheorem{corollary}{Corollary}
\newtheorem{lemma}{Lemma}

\newtheorem{assumption}{Assumption}
\newtheorem*{assumptiona}{Assumption (A')}

\newtheorem*{doeblin}{Doeblin condition}

\theoremstyle{definition}
\newtheorem{definition}{Definition}

\theoremstyle{remark}
\newtheorem{remark}{Remark}

\def\L{\mathbb{L}}
\def\E{\mathbb{E}}
\def\P{\mathbb{P}}
\def\R{\mathbb{R}}
\def\Q{\mathbb{Q}}
\def\N{\mathbb{N}}
\def\1{\mathbbm{1}}

\def\Z{\mathbb{Z}}

\def\cE{{\mathcal E}}
\def\cF{{\mathcal F}}

\def\cC{\mathcal{C}}
\def\cM{\mathcal{M}}
\def\cB{\mathcal{B}}
\def\cA{\mathcal{A}}
\def\cT{\mathcal{T}}
\def\cN{\mathcal{N}}
\def\Id{\text{Id}}

\title{An ergodic theorem for asymptotically periodic time-inhomogeneous Markov processes, with application to quasi-stationarity with moving boundaries.}
\author{William Oçafrain$^1$}
\date{\today}

\begin{document}

\maketitle

\begin{abstract}
    This paper deals with ergodic theorems for particular time-inhomogeneous Markov processes, whose time-inhomogeneity is asymptotically periodic. Under a Lyapunov/minorization condition, it is shown that, for any measurable  bounded function $f$, the time average $\frac{1}{t} \int_0^t f(X_s)ds$ converges in $\L^2$ towards a limiting distribution, starting from any initial distribution for the process $(X_t)_{t \geq 0}$. This convergence can be improved to an almost sure convergence under an additional assumption on the initial measure. This result will be then applied to show the existence of a quasi-ergodic distribution for processes absorbed by an asymptotically periodic moving boundary, satisfying a conditional Doeblin condition. 
    \end{abstract}


\footnotetext[1]{Université de Lorraine, CNRS, Inria, IECL, UMR 7502, F-54000, Nancy, France.\\E-mail: w.ocafrain@hotmail.fr}

\section*{Notation}
Throughout we shall use the following notation:
\begin{itemize}
    \item $\N = \{1,2, \ldots, \}$ and $\Z_+ = \{0\} \cup \N$.
    \item $\cM_1(E)$ denotes the space of the probability measures whose support is included in $E$.
    \item $\cB(E)$ denotes the set of the measurable bounded functions defined on $E$.
    \item $\cB_1(E)$: denotes the set of the measurable functions $f$ defined on $E$ such that $\|f\|_\infty \leq 1$.
    \item \textcolor{black}{For all $\mu \in \cM_1(E)$ and $p \in \N$, $\L^p(\mu)$ denotes the set of the measurable functions $f : E \mapsto \R$ such that $\int_E |f(x)|^p \mu(dx) < + \infty$.}
    \item For any $\mu \in \cM_1(E)$ and $f \in \L^1(\mu)$, denote 
    $$\mu(f) := \int_E f(x) \mu(dx).$$
\item For any positive function $\psi$,
$$\cM_1(\psi) := \{\mu \in \cM_1(E) : \mu(\psi) < + \infty\}.$$
    \item $\Id$ denotes the Identity operator.
\end{itemize}

\section{Introduction}

In general, an ergodic theorem for a Markov process $(X_t)_{t \geq 0}$ and probability measure $\pi$ refers to the almost sure convergence \begin{equation} \label{ergodic-theorem}\frac{1}{t} \int_0^t f(X_s)ds \underset{t \to \infty}{\longrightarrow} \pi(f),~~~~\forall f \in \L^1(\pi). \end{equation}  
In the time-homogeneous setting, such an ergodic theorem holds for positive Harris recurrent Markov processes with the limiting distribution $\pi$ corresponding to an invariant measure for the underlying Markov process. For time-inhomogeneous Markov processes, such a result does not hold in general (in particular the notion of invariant measure is in general not well-defined), 
except for specific types of time-inhomogeneity such as \textit{periodic time-inhomogeneous Markov processes}, defined as time-inhomogeneous Markov processes for which there exists $\gamma > 0$ such that, for any $s \leq t$, $k \in \Z_+$ and $x$, 
\begin{equation}
\label{periodicity}
\P[X_t \in \cdot | X_s = x] = \P[X_{t+k\gamma} \in \cdot | X_{s+k \gamma} = x].\end{equation}
In other words, a time-inhomogeneous Markov process is periodic when the transition law between any times $s$ and $t$ remains unchanged when the time interval $[s,t]$ is shifted by a multiple of the period $\gamma$. In particular, this implies that, for any $s \in [0,\gamma)$, the Markov chain $(X_{s+n \gamma})_{n \in \Z_+}$ is time-homogeneous.  
This fact allowed Höpfner et al. (in \cite{HK2010,hopfner2016ergodicity,hopfner2016ergodicity1}) to show that, if the skeleton Markov chain $(X_{n \gamma})_{n \in \Z_+}$ is Harris recurrent, then the chains $(X_{s + n \gamma})_{n \in \Z_+}$, for all $s \in [0,\gamma)$, are also Harris recurrent and $$\frac{1}{t} \int_0^t f(X_s)ds \underset{t \to \infty}{\longrightarrow} \frac{1}{\gamma} \int_0^\gamma \pi_s(f)ds,~~~~\text{almost surely, from any initial measure,}$$
where $\pi_s$ is the invariant measure for $(X_{s+n \gamma})_{n \in \Z_+}$.

This paper aims to prove a similar result for time-inhomogeneous Markov processes said to be \textit{asymptotically periodic}. Roughly speaking (a precise definition of which will be explicitly given later), an asymptotically periodic Markov process is such that, given a time interval $T \geq 0$, its transition law on the interval $[s,s+T]$ is asymptotically \textit{"close to"} the one, on the same interval, of a periodic time-inhomogeneous Markov process called an \textit{auxiliary Markov process}, when $s \to \infty$. This definition is very similar to the notion of \textit{asymptotic homogeneization}, defined as follows in \cite[Subsection 3.3]{BCG2017}: a time-inhomogeneous Markov process $(X_t)_{t \geq 0}$ is said to be \textit{asymptotically homogeneous} if there exists a time-homogeneous Markovian semigroup $(Q_t)_{t \geq 0}$ such that, for all $s \geq 0$, \begin{equation}\lim_{t \to \infty} \sup_{x} \left\| \P[X_{t+s} \in \cdot | X_t = x] - \delta_x Q_s\right\|_{TV} = 0,\end{equation}
where, for two positive measures with finite mass $\mu_1$ and $\mu_2$, $\| \mu_1 - \mu_2\|_{TV}$ is the \textit{total variation distance} between $\mu_1$ and $\mu_2$:
\begin{equation}\| \mu_1 - \mu_2\|_{TV} := \sup_{f \in \cB_1(E)} |\mu_1(f) - \mu_2(f)|.\end{equation}
In particular, it is well-known (see \cite[\textcolor{black}{Theorem 3.11}]{BCG2017}) that, under this suitable additional conditions, an asymptotically homogeneous Markov process converges towards a probability measure which is invariant for $(Q_t)_{t \geq 0}$. It is similarly expected that an asymptotically periodic process has the same asymptotic properties as a periodic Markov process; in particular an ergodic theorem holds for the asymptotically periodic process.

The main result of this paper provides for an asymptotically periodic Markov process to satisfy
\begin{equation}\label{ergodic}\frac{1}{t} \int_0^t f(X_s)ds \underset{t \to \infty}{\overset{\L^2(\P_{0,\mu})}{\xrightarrow{\hspace*{1.2cm}}}} \frac{1}{\gamma} \int_0^\gamma \beta_s(f) ds,~~~~\forall f \in \cB(E), \forall \mu \in \cM_1(E),\end{equation}
whereby $\P_{0,\mu}$ is a probability measure under which $X_0 \sim \mu$, and
where $\beta_s$ is the limiting distribution of the skeleton Markov chain $(X_{s+n \gamma})_{n \in \Z_+}$,
if it satisfies a Lyapunov-type condition, a local Doeblin condition (defined further in Section \ref{2}), and is such that its auxiliary process satisfies a Lyapunov/minorization condition.

Furthermore, this convergence result holds almost surely if a Lyapunov function of the process $(X_t)_{t \geq 0}$, denoted by $\psi$, is integrable with respect to the initial measure:
$$\frac{1}{t} \int_0^t f(X_s)ds \underset{t \to \infty}{\overset{\P_{0,\mu}-\text{almost surely}}{\xrightarrow{\hspace*{2cm}}}} \frac{1}{\gamma} \int_0^\gamma \beta_s(f) ds,~~~~\forall \mu \in \cM_1(\psi).$$
This will be more precisely stated and proved in Section \ref{2}.

\textcolor{black}{The main motivation of this paper is then to deal with \textit{quasi-stationarity with moving boundaries}, that is the study of asymptotic properties for the process $X$, conditioned not to reach some moving subset of the state space. In particular, such a study is motivated by models such as those presented in \cite{CCG2016}, studying Brownian particles absorbed by cells whose volume may vary over time.} 
 
 Quasi-stationarity with moving boundaries has been studied in particular in \cite{OCAFRAIN2017,ocafrain2018}, where a \textit{"conditional ergodic theorem"} (see further the definition of a \textit{quasi-ergodic distribution}) has been shown when the absorbing boundaries move periodically. In this paper, we show that a similar result holds  when the boundary is asymptotically periodic, assuming that the process satisfies a conditional Doeblin condition (see Assumption (A')). This will be dealt with in Section \ref{3}. 
 
The paper will be concluded by using these results in two examples: an ergodic theorem for an asymptotically periodic Ornstein-Uhlenbeck process, and the existence of a unique quasi-ergodic distribution for a Brownian motion confined between two symmetric asymptotically periodic functions.

\section{Ergodic theorem for asymptotically periodic time-inhomogeneous semigroup.}
\label{2}

\subsubsection*{Asymptotic periodicity: the definition.}


Let $(E,\cE)$ be a measurable space. 
Consider $\{(E_t, \cE_t)_{t \geq 0}, (P_{s,t})_{s \leq t}\}$ a Markovian time-inhomogeneous semigroup, giving a family of measurable subspaces of $(E, \cE)$, denoted by $(E_t, \cE_t)_{t \geq 0}$, and a family of linear operator $(P_{s,t})_{s \leq t}$, with $P_{s,t} : \cB(E_t) \to \cB(E_s)$, satisfying for any $r \leq s \leq t$,
$$P_{s,s} = \Id,~~~~~~~~P_{s,t}\1_{E_t} = \1_{E_s},~~~~~~~~P_{r,s}P_{s,t} = P_{r,t}.$$
In particular, associated to $\{(E_t, \cE_t)_{t \geq 0}, (P_{s,t})_{s \leq t}\}$ is a Markov process $(X_t)_{t \geq 0}$ and a family of probability measures $(\P_{s,x})_{s \geq 0, x \in E_s}$  such that, for any $s \leq t$, $x \in E_s$ and $A \in \cE_t$,
$$\P_{s,x}[X_t \in A] = P_{s,t}\1_{A}(x).$$
We denote $\P_{s,\mu} := \int_{E_s} \P_{s,x} \mu(dx)$ for any probability measure $\mu$ supported on $E_s$. We also denote by $\E_{s,x}$ and $\E_{s,\mu}$ the expectations associated to $\P_{s,x}$ and $\P_{s,\mu}$ respectively. 
Finally, the following notation will be used for $\mu \in \cM_1(E_s)$, $s \leq t$ and $f \in \cB(E_t)$,
$$\mu P_{s,t} f := \E_{s,\mu}[f(X_t)],~~~~~~\mu P_{s,t} := \P_{s,\mu}[X_t \in \cdot].$$ 
The periodicity of a time-inhomogeneous semigroup is defined as follows. We say a semigroup $\{(F_t, \cF_t)_{t \geq 0},(Q_{s,t})_{s \leq t}\}$ is \textit{$\gamma$-periodic} (for $\gamma > 0$), if, for any $s \leq t$,
$$(F_t, \cF_t) = (F_{t+k\gamma}, \cF_{t+k\gamma}),~~~~Q_{s,t} = Q_{s+k \gamma,t+k\gamma},~~~~\forall k \in \Z_+.$$ 
It is now possible to define an \textit{asymptotically periodic semigroup}.
\begin{definition}[Asymptotically periodic semigroups]
\label{asymptotic-periodicity}
A time-inhomogeneous semigroup $\{(E_t, \cE_t)_{t \geq 0},(P_{s,t})_{s \leq t}\}$ is said to be \textit{asymptotically periodic} if (for some $\gamma > 0$) there exists a $\gamma$-periodic semigroup $\{(F_t, \cF_t)_{t \geq 0},(Q_{s,t})_{s \leq t}\}$ and two families of functions $(\psi_s)_{s \geq 0}$ and $(\tilde{\psi}_s)_{s \geq 0}$  such that $\tilde{\psi}_{s+\gamma} = \tilde{\psi}_s$ for all $s \geq 0$, and for any $s \in [0, \gamma)$: \begin{enumerate} \item $\bigcup_{k=0}^\infty \bigcap_{l \geq k} E_{s+l\gamma} \cap F_s \ne \emptyset$; \item there exists $x_s \in \bigcup_{k=0}^\infty \bigcap_{l \geq k} E_{s+l\gamma} \cap F_s$ such that, for any $n \in \Z_+$, 
\begin{equation} \label{asympt-hom}  \|\delta_{x_s} P_{s+k \gamma, s+(k+n)\gamma}[\psi_{s+(k+n)\gamma} \times \cdot] - \delta_{x_s} Q_{s, s + n \gamma}[\tilde{\psi}_{s} \times \cdot]\|_{TV} \underset{k \to \infty}{\longrightarrow} 0.\end{equation}
\end{enumerate}
The semigroup $\{(F_t, \cF_t)_{t \geq 0},(Q_{s,t})_{s \leq t}\}$ is then called the \textit{auxiliary semigroup of $(P_{s,t})_{s \leq t}$}. 
\end{definition}

When $\psi_s = \tilde{\psi}_s = \1$ for all $s \geq 0$, we say that the semigroup $(P_{s,t})_{s \leq t}$ is \textit{asymptotically periodic in total variation.} By extension, we will say that the process $(X_t)_{t \geq 0}$ is asymptotically periodic (in total variation) if the associated semigroup $\{(E_t, \cE_t)_{t \geq 0},(P_{s,t})_{s \leq t}\}$ is asymptotically periodic (in total variation).

\textcolor{black}{In what follows, the functions $(\psi_s)_{s \geq 0}$ and $(\tilde{\psi}_s)_{s \in [0,\gamma)}$ will play the role of Lyapunov functions (that is to say satisfying Assumption \ref{lyap-def}  (ii) below) for the semigroups $(P_{s,t})_{s \leq t}$ and $(Q_{s,t})_{s \leq t}$ respectively. The introduction of these functions in the definition of asymptotically periodic semigroups will allow us to establish an ergodic theorem for processes satisfying the Lyapunov/minorization conditions written below.}
 
 \subsubsection*{Lyapunov/minorization conditions.}
 
 The main assumption of Theorem \ref{thm}, which will be provided later, will be that  the asymptotically periodic Markov process satisfies the following assumption.
 \begin{assumption}
\label{lyap-def}
There exist $t_1 \geq 0$, $n_0 \in \N$, $c > 0$, $\theta \in (0,1)$, a family of \textcolor{black}{measurable sets $(K_t)_{t \geq 0}$ such that $K_t \subset E_t$ for all $t \geq 0$}, a family of probability measures $\left(\nu_{s}\right)_{s \geq 0}$ on $(K_{s})_{s \geq 0}$,  and a family of functions $(\psi_s)_{s \geq 0}$, all lower-bounded by $1$, such that:
 \begin{enumerate}[(i)]
 \item For any $s \geq 0$, $x \in K_s$ and $n \geq n_0$,
     $$\delta_x P_{s,s+n t_1} \geq c \nu_{s+nt_1}.$$
     \item For any $s \geq 0$, $$P_{s,s+t_1} \psi_{s+t_1} \leq \theta \psi_s + C \1_{K_s}.$$
     \item For any $s \geq 0$ and $t \in [0,t_1)$,
     $$P_{s,s+t} \psi_{s+t} \leq C \psi_s.$$
 \end{enumerate}
 \end{assumption}
\textcolor{black}{When a semigroup $(P_{s,t})_{s \leq t}$ satisfies Assumption \ref{lyap-def} as stated above, we will say that the family of functions $(\psi_s)_{s \geq 0}$ are \textit{Lyapunov functions} for the semigroup $(P_{s,t})_{s \leq t}$.}
 In particular, under $(ii)$ and $(iii)$, it is easy to prove that for any $s \leq t$,
 \begin{equation}
     \label{maj}
     P_{s,t} \psi_t \leq C \left(1 + \frac{C}{1-\theta}\right) \psi_s.
 \end{equation}
 We remark in particular that Assumption \ref{lyap-def} implies an \textit{exponential weak ergodicity in $\psi_t$-distance}, that is we have the existence of two constants $C' > 0$ and $\kappa > 0$ such that, for all $s \leq t$  and for all probability measures $\mu_1, \mu_2 \in \cM_1(E_s)$,
 \begin{equation}
     \label{mixing}
     \| \mu_1 P_{s,t} - \mu_2 P_{s,t} \|_{\psi_t} \leq C' [\mu_1(\psi_s) + \mu_2(\psi_s)] e^{-\kappa (t-s)},
 \end{equation}
 whereby, for a given function $\psi$, $\| \mu - \nu \|_\psi$ is the $\psi$-distance defined to be 
 $$\| \mu - \nu \|_{\psi} := \sup_{|f| \leq \psi} \left| \mu(f) - \nu(f) \right|,~~~~\forall \mu, \nu \in \cM_1(\psi).$$
 In particular, when $\psi = \1$ for all $t \geq 0$, the $\psi$-distance is the total variation distance. If we have weak ergodicity \eqref{mixing} in the time-homogeneous setting (see in particular \cite{hairer2011yet}), the proof of \cite[Theorem 1.3.]{hairer2011yet} can be adapted to a general time-inhomogeneous framework (see for example \cite[Subsection 9.5]{CV2017c} 
). 

 
\subsubsection*{The main theorem and proof.}
 
The main result of this paper is the following.
\begin{theorem}
\label{thm}
Let $\{ (E_t, \cE_t)_{t \geq 0},(P_{s,t})_{s \leq t}, (X_t)_{t \geq 0}, (\P_{s,x})_{s \geq 0, x \in E_s}\}$ be an asymptotically $\gamma$-periodic time-inhomogeneous Markov process, with $\gamma > 0$, and denote by $\{(F_t, \cF_t)_{t \geq 0},(Q_{s,t})_{s \leq t}\}$ its periodic auxiliary semigroup. Also, denote by $(\psi_s)_{s \geq 0}$ and $(\tilde{\psi}_s)_{s \geq 0}$ the two families of functions as defined in Definition \ref{asymptotic-periodicity}.
Assume moreover that:
\begin{enumerate}
    \item The semigroups $(P_{s,t})_{s \leq t}$ and $(Q_{s,t})_{s \leq t}$ satisfy Assumption \ref{lyap-def}, with $(\psi_s)_{s \geq 0}$ and $(\tilde{\psi}_s)_{s \geq 0}$ as Lyapunov functions respectively.
    \item For any $s \in [0,\gamma)$, $(\psi_{s+n \gamma})_{n \in \Z_+}$ converges pointwisely to $\tilde{\psi}_s$. 
\end{enumerate}
Then, for any $\mu \in \cM_1(E_0)$ such that $\mu(\psi_0) < + \infty$, \begin{equation} \label{tv} \left\|\frac{1}{t} \int_0^t \mu P_{0,s}[\psi_s \times \cdot] ds - \frac{1}{\gamma} \int_0^\gamma \beta_{\gamma} Q_{0,s}[\tilde{\psi}_s \times \cdot] ds \right\|_{TV} \underset{t \to \infty}{{\longrightarrow}} 0, \end{equation}
whereby $\beta_\gamma \in \cM_1(F_0)$ is the unique invariant probability measure of the skeleton semigroup $(Q_{0,n \gamma})_{n \in \Z_+}$ satisfying $\beta_\gamma(\tilde{\psi}_0) < + \infty$. 
Moreover, for any  $f \in \cB(E)$ we have:
\begin{enumerate}
\item For any $\mu \in \cM_1(E_0)$,
\begin{equation}
\label{l2}
    \E_{0,\mu}\left[\left|\frac{1}{t} \int_0^t f(X_s)ds -  \frac{1}{\gamma} \int_0^\gamma \beta_\gamma Q_{0,s}f ds\right|^2\right] \underset{t \to \infty}{\longrightarrow} 0.
\end{equation}
\item If moreover $\mu(\psi_0) < + \infty$, then
\begin{equation}\label{as}\frac{1}{t} \int_0^t f(X_s) ds  \underset{t \to \infty}{\longrightarrow} \frac{1}{\gamma} \int_0^\gamma \beta_\gamma Q_{0,s}f ds,~~~~\P_{0,\mu}-\text{almost surely.}\end{equation}
\end{enumerate}
\end{theorem}

\begin{remark}
When Assumption \ref{lyap-def} hold for $K_s = E_s$ for any $s$, then the condition $(i)$ in Assumption \ref{lyap-def} implies \textit{Doeblin condition}.
 \end{remark}
 \begin{doeblin}
 There exists $t_0 \geq 0$, $c > 0$ and a family of probability measure $(\nu_t)_{t \geq 0}$ on $(E_t)_{t \geq 0}$ such that, for any $s \geq 0$ and $x \in E_s$,
 \begin{equation} \label{doeblin}\delta_x P_{s,s+t_0} \geq c \nu_{s+t_0}.\end{equation}
\end{doeblin}
In fact, if we assume that Assumption \ref{lyap-def} (i) holds for $K_s = E_s$, Doeblin condition holds by setting $t_0 := n_0 t_1$. Conversely, the Doeblin condition implies the conditions (i)-(ii)-(iii) with $K_s = E_s$ and $\psi_s = \1_{E_s}$ for all $s \geq 0$, so that these conditions are equivalent. In fact, (ii) and (iii) straightforwardly hold true for $(K_s)_{s \geq 0} = (E_s)_{s \geq 0}$, $(\psi_s)_{s \geq 0} = (\1_{E_s})_{s \geq 0}$, $C = 1$, any $\theta \in (0,1)$ and any $t_1 \geq 0$. Setting $t_1 = t_0$ and $n_0 = 1$, the Doeblin condition entails that, for any $s \in [0,t_1)$,
$$\delta_x P_{s,s+t_1} \geq c \nu_{s+t_1},~~~~\forall x \in E_s.$$
Integrating this inequality over $\mu \in \cM_1(E_s)$, one obtains
$$\mu P_{s,s+t_1} \geq c \nu_{s+t_1},~~~~\forall s \in [0,t_1), \forall \mu \in \cM_1(E_s).$$
Then, by the Markov property, for all $s \in [0,t_1)$, $x \in E_s$ and $n \in \N$, we have
$$\delta_x P_{s,s+nt_1} = (\delta_x P_{s,s+(n-1)t_1})P_{s + (n-1)t_1, s + nt_1} \geq c \nu_{s+nt_1},$$
which is (i).  

 Theorem \ref{thm} then entails the following corollary.
 \begin{corollary}
 \label{corollary}
 Let $(X_t)_{t \geq 0}$ be asymptotically $\gamma$-periodic in total variation distance. If $(X_t)_{t \geq 0}$ and its auxiliary semigroup satisfy a Doeblin condition, then the convergence \eqref{l2} is improved to
 \begin{equation*}
    \sup_{\mu \in \cM_1(E_0)} \sup_{f \in \cB_1(E)} \E_{0,\mu}\left[\left|\frac{1}{t} \int_0^t f(X_s)ds -  \frac{1}{\gamma} \int_0^\gamma \beta_\gamma Q_{0,s}f ds\right|^2\right] \underset{t \to \infty}{\longrightarrow} 0.
\end{equation*}
Moreover, the almost sure convergence \eqref{as} holds for any initial measure $\mu$.
 \end{corollary}
 
 \begin{remark}
We also note that, if the convergence \eqref{asympt-hom} holds for all $x \in \bigcup_{k=0}^\infty \bigcap_{l \geq k} E_{s+l \gamma} \cap F_s$, then this implies \eqref{asympt-hom} and therefore the pointwise convergence of $(\psi_{s+n \gamma})_{n \in \Z_+}$ to $\tilde{\psi}_s$ (by taking $n = 0$ in \eqref{asympt-hom}). 
 \end{remark}

\begin{proof}[Proof of Theorem \ref{thm}] The proof is divided into five steps. \medskip

\textit{First step.} Since the auxiliary semigroup $(Q_{s,t})_{s \leq t}$ satisfies Assumption \ref{lyap-def} with $(\tilde{\psi}_s)_{s \geq 0}$ as Lyapunov functions, the time-homogeneous semigroup $(Q_{0,n\gamma})_{n \in \Z_+}$ satisfies assumptions 1 and 2 of \cite{hairer2011yet}, which we now recall (using our notation).
\begin{assumption}[Assumption 1, \cite{hairer2011yet}] \label{assumption-1} There exists $V : F_0 \to [0,+\infty)$, $n_1 \in \N$ and constants $K \geq 0$ and $\kappa \in (0,1)$ such that
$$Q_{0,n_1\gamma} V \leq \kappa V + K.$$
\end{assumption}
\begin{assumption}[Assumption 2, \cite{hairer2011yet}] \label{assumption-2}
There exists a constant $\alpha \in (0,1)$ and a probability measure $\nu$ such that 
$$\inf_{x \in \cC_R} \delta_x Q_{0,n_1\gamma} \geq \alpha \nu(\cdot),$$
with $\cC_R := \{x \in F_0 : V(x) \leq R \}$ for some $R > 2 K/(1-\kappa)$, whereby $n_1$, $K$ and $\kappa$ are the constants from Assumption \ref{assumption-1}.
\end{assumption}
In fact, since $(Q_{s,t})_{s \leq t}$ satisfies (ii) and (iii) of Assumption \ref{lyap-def}, there exists $C > 0$, $\theta \in (0,1)$, $t_1 \geq 0$ and $(K_s)_{s \geq 0}$ such that 
\begin{equation}\label{(ii)-Q}Q_{s,s+t_1} \tilde{\psi}_{s+t_1} \leq \theta \tilde{\psi}_s + C \1_{K_s},~~~~\forall s \geq 0,\end{equation}
and 
$$Q_{s,s+t}\tilde{\psi}_{s+t} \leq C \tilde{\psi}_s,~~~~\forall s \geq 0, \forall t \in [0,t_1).$$
We let $n_2 \in \N$ be such that $\theta^{n_2} C (1 + \frac{C}{1-\theta}) < 1$. By \eqref{(ii)-Q} and recalling that $\tilde{\psi}_t = \tilde{\psi}_{t+\gamma}$ for all $t \geq 0$, one has for any $s \geq 0$ and $n \in \N$,
\begin{equation}\label{intermediaire}Q_{s,s+nt_1}\tilde{\psi}_{s+nt_1} \leq \theta^n \tilde{\psi}_s + \frac{C}{1 - \theta}.\end{equation}
Thus, for all $n_1 \geq \lceil \frac{n_2t_1}{\gamma} \rceil$, 
$$Q_{0,n_1\gamma}\tilde{\psi}_0 = Q_{0,n_1 \gamma - n_2t_1} Q_{n_1 \gamma - n_2t_1, n_1 \gamma} \tilde{\psi}_{n_1 \gamma} \leq \theta^{n_2} Q_{0,n_1 \gamma -n_2t_1} \tilde{\psi}_{n_1\gamma-n_2t_1} + \frac{C}{1-\theta} \leq \theta^{n_2}C(1 + \frac{C}{1-\theta}) \tilde{\psi}_0 + \frac{C}{1 - \theta},$$
where we successively used the property of semigroup of $(Q_{s,t})_{s \leq t}$, \eqref{intermediaire} and \eqref{maj} applied to $(Q_{s,t})_{s \leq t}$. Hence one has Assumption \ref{assumption-1} by setting $V = \tilde{\psi}_0$, $\kappa := \theta^{n_2}C(1 + \frac{C}{1-\theta})$ and $K := \frac{C}{1 - \theta}$. 

We now prove Assumption \ref{assumption-2}. To this end, we introduce a Markov process $(Y_t)_{t \geq 0}$ and a family of probability measures $(\hat{\P}_{s,x})_{s \geq 0, x \in F_s}$ such that
$$\hat{\P}_{s,x}(Y_t \in A) = Q_{s,t}\1_A(x),~~~~\forall s \leq t, x \in F_s, A \in \cF_t.$$
In what follows, for all $s \geq 0$ and $x \in F_s$, we will use the notation $\hat{\E}_{s,x}$ for the expectation associated to $\hat{\P}_{s,x}$.
Moreover, we define $$T_K := \inf\{n \in \Z_+ : Y_{n t_1} \in K_{nt_1}\}.$$
Then, using \eqref{(ii)-Q} recursively, for all $k \in \N$, $R > 0$ and $x \in \cC_R$ (recalling that $\cC_R$ is defined in the statement of Assumption \ref{assumption-2}) we have
\begin{align*}\hat{\E}_{0,x}[\tilde{\psi}_{kt_1}(Y_{kt_1})\1_{T_K > k}] &= \hat{\E}_{0,x}[\1_{T_K > k-1} \hat{\E}_{(k-1)t_1,Y_{(k-1)t_1}}(\tilde{\psi}_{kt_1}(Y_{kt_1})\1_{T_K > k})] \\&\leq \theta \hat{\E}_{0,x}[\tilde{\psi}_{(k-1)t_1}(Y_{(k-1)t_1})\1_{T_K > k-1}] \leq \theta^{k} \tilde{\psi}_0(x) \leq R \theta^k.\end{align*}
Since $\tilde{\psi}_{kt_1} \geq 1$ for all $k \in \Z_+$, then for all $x \in \cC_R$, for all $k \in \Z_+$,
$$\hat{\P}_{0,x}(T_k > k) \leq R \theta^k.$$
In particular, there exists $k_0 \geq n_0$ such that, for all $k \geq k_0 - n_0$, 
$$\hat{\P}_{0,x}(T_K > k) \leq \frac{1}{2}.$$
Hence, for all $x \in \cC_R$,
\begin{align*}
\delta_x Q_{0,k_0t_1} = \hat{\P}_{0,x}(Y_{k_0t_1} \in \cdot) \geq &\sum_{i=0}^{k_0-n_0} \hat{\E}_{0,x}(\1_{T_k = i} \hat{\P}_{it_1,X_{it_1}}(Y_{k_0t_1} \in \cdot)) \\
&\geq c \sum_{i=0}^{k_0-n_0} \hat{\E}_{0,x}(\1_{T_K = i}) \times \nu_{k_0t_1} \\
&= c \hat{\P}_{0,x}(T_K \leq k_0 - n_0) \nu_{k_0t_1} \\
&\geq \frac{c}{2} \nu_{k_0t_1}.
\end{align*}
Hence, for all $n_1 \geq \lceil \frac{k_0t_1}{\gamma} \rceil$, for all $x \in \cC_R$,
$$\delta_x Q_{0,k_0t_1}Q_{k_0t_1,n_1\gamma} \geq \frac{c}{2} \nu_{k_0t_1} Q_{k_0t_1, n_1 \gamma}.$$
Thus, Assumption \ref{assumption-2} is satisfied taking $n_1 := \lceil \frac{n_2 t_1}{\gamma}\rceil \lor \lceil \frac{k_0t_1}{\gamma} \rceil$, $\alpha := \frac{c}{2}$ and $\nu(\cdot) := \nu_{k_0t_1} Q_{k_0t_1,n_1\gamma}$. 
 
Then, by \cite[Theorem 1.2]{hairer2011yet}, Assumptions \ref{assumption-1} and \ref{assumption-2} imply that $Q_{0,n_1 \gamma}$ admits a unique invariant probability measure $\beta_\gamma$. Furthermore, there exists constants $C > 0$ and $\delta \in (0,1)$ such that, for all $\mu \in \cM_1(F_0)$,
\begin{equation}\label{convergence-exponentielle}\| \mu Q_{0,nn_1\gamma} - \beta_\gamma\|_{\tilde{\psi}_0} \leq C \mu(\tilde{\psi}_0) \delta^n.\end{equation}
Since $\beta_\gamma$ is the unique invariant probability measure of $Q_{0,n_1 \gamma}$ and noting that $\beta_\gamma Q_{0,\gamma}$ is invariant for $Q_{0,n_1 \gamma}$, we deduce that $\beta_\gamma$ is the unique invariant probability measure for $Q_{0,\gamma}$ and, by \eqref{convergence-exponentielle}, for all $\mu$ such that $\mu(\tilde{\psi}_0) < + \infty$, 
$$\|\mu Q_{0,n\gamma} - \beta_\gamma\|_{\tilde{\psi}_0} \underset{n \to \infty}{\longrightarrow} 0.$$
Now, for any $s \geq 0$, note that $\delta_x Q_{s,\lceil \frac{s}{\gamma} \rceil \gamma} \tilde{\psi}_0 < + \infty$ for all $x \in F_s$ (this is a consequence of \eqref{maj} applied to the semigroup $(Q_{s,t})_{s \leq t}$), and therefore, taking $\mu = \delta_x Q_{s,\lceil \frac{s}{\gamma} \rceil \gamma}$ in the above convergence, 
$$\|\delta_x Q_{s,n\gamma} - \beta_\gamma\|_{\tilde{\psi}_0} \underset{n \to \infty}{\longrightarrow} 0$$
 for all $x \in F_s$. Hence, since $Q_{n\gamma, n \gamma +s}\tilde{\psi}_s \leq C (1 + \frac{C}{1-\theta}) \tilde{\psi}_{n \gamma}$ by \eqref{maj}, we conclude from the above convergence that 
\begin{equation} \label{decay}\| \delta_{x} Q_{s,s+n\gamma} - \beta_\gamma Q_{0,s} \|_{\tilde{\psi}_s} \leq C (1 + \frac{C}{1-\theta}) \|\delta_x Q_{s,n\gamma} - \beta_\gamma\|_{\tilde{\psi}_0} \underset{n \to \infty}{\longrightarrow} 0.\end{equation}
Moreover, $\beta_\gamma(\tilde{\psi}_0) < + \infty$.

\textit{Second step.}  The first part of this step (until the equality \eqref{bansaye}) is inspired by the proof of \cite[Theorem 3.11]{BCG2017}.

We fix $s \in [0,\gamma]$. Without loss of generality, we assume that $\bigcap_{l \geq 0} E_{s+l \gamma} \cap F_s \ne \emptyset$. Then, by Definition \ref{asymptotic-periodicity}, there exists $x_s \in \bigcap_{l \geq 0} E_{s+l \gamma} \cap F_s$ such that for any $n \geq 0$,
\begin{equation*}  \|\delta_{x_s} P_{s+k \gamma, s+(k+n)\gamma}[\psi_{s+(k+n)\gamma} \times \cdot] - \delta_{x_s} Q_{s, s + n \gamma}[\tilde{\psi}_{s} \times \cdot]\|_{TV} \underset{k \to \infty}{\longrightarrow} 0,\end{equation*}
which implies by \eqref{decay} that
\begin{equation}
    \label{asymptotic-homogeneization}
    \lim_{n \to \infty} \lim_{k \to \infty} \|\delta_{x_s} P_{s+k \gamma, s+(k+n)\gamma}[\psi_{s+(k+n)\gamma} \times \cdot] - \beta_\gamma Q_{0,s}[\tilde{\psi}_s \times \cdot] \|_{TV} = 0.
\end{equation}
Then, by the Markov property, \eqref{mixing}, and \eqref{maj}, one obtains that, for any $k,n \in \N$ and $x \in \bigcap_{l \geq 0} E_{s+l \gamma}$,
\begin{align}
    \|\delta_{x} P_{s,s+(k+n)\gamma} - \delta_{x} P_{s+k\gamma, s+(k+n)\gamma} \|_{\psi_{s+(k+n)\gamma}} &= \|\left(\delta_{x} P_{s,s+k \gamma} \right) P_{s+k\gamma,s+(k+n)\gamma} - \delta_{x} P_{s+k\gamma, s+(k+n)\gamma} \|_{\psi_{s+(k+n)\gamma}} \notag \\ &\leq C' [P_{s,s+k \gamma} \psi_{s+k \gamma}(x) + \psi_{s+k \gamma}(x)] e^{- \kappa \gamma n}  \leq  C''[\psi_s(x) +  \psi_{s+k \gamma}(x)] e^{-\kappa \gamma n}, \label{premiere-etape}
\end{align}
 whereby $C'' := C' \left(C \left(1 + \frac{C}{1-\theta}\right) \lor 1\right)$. Then, for any $k,n \in \N$,
    \begin{multline}
        \label{pour-la-remark}
         \| \delta_{x_s} P_{s,s+(k+n)\gamma}[\psi_{s+(k+n)\gamma} \times \cdot] - \beta_\gamma Q_{0,s} [\tilde{\psi}_{s} \times \cdot] \|_{TV} \\ \leq   C''[\psi_s(x) +  \psi_{s+k \gamma}(x)] e^{-\kappa \gamma n}  +\|\delta_{x_s} P_{s+k \gamma, s+(k+n)\gamma}[\psi_{s+(k+n)\gamma} \times \cdot] - \beta_\gamma Q_{0,s}[\tilde{\psi}_s \times \cdot] \|_{TV},
    \end{multline}
    which by \eqref{asymptotic-homogeneization} and the pointwise convergence of $(\psi_{s+k \gamma})_{k \in \Z_+}$ implies that
    \begin{equation} \label{bansaye}\lim_{n \to \infty} \| \delta_{x_s} P_{s,s+n\gamma}[\psi_{s+n\gamma} \times \cdot] - \beta_\gamma Q_{0,s}[\tilde{\psi}_s \times \cdot] \|_{TV}  = \lim_{n \to \infty} \limsup_{k \to \infty}  \| \delta_{x_s} P_{s,s+(k+n)\gamma}[\psi_{s+(k+n)\gamma} \times \cdot] - \beta_\gamma Q_{0,s} [\tilde{\psi}_{s} \times \cdot] \|_{TV}  = 0.\end{equation}
The weak ergodicity \eqref{mixing} implies therefore that the previous convergence actually holds for any initial distribution $\mu \in \cM_1(E_0)$ satisfying $\mu(\psi_0) < + \infty$, so that
\begin{equation}
    \label{jjj}
    \left\| \mu P_{0,s+n \gamma}[\psi_{s+n \gamma} \times \cdot] - \beta_\gamma Q_{0,s}[\tilde{\psi}_s \times \cdot] \right\|_{TV} \underset{n \to \infty}{\longrightarrow} 0.
\end{equation}
Since $\|\mu P_{0,s+n\gamma}[\psi_{s+n\gamma} \times \cdot] - \beta_\gamma Q_{0,s}[\tilde{\psi}_s \times \cdot]\|_{TV} \leq 2$ for all $\mu \in \cM_1(E_0)$, $s \geq 0$ and $n \in \Z_+$, \eqref{jjj} and Lebesgue's dominated convergence theorem implies that
$$\frac{1}{\gamma} \int_0^\gamma \|\mu P_{0,s+n\gamma}[\psi_{s+n\gamma} \times \cdot] - \beta_\gamma Q_{0,s}[\tilde{\psi}_s \times \cdot]\|_{TV}ds \underset{n \to \infty}{\longrightarrow} 0,$$
which implies that 
$$\left\|\frac{1}{\gamma}\int_0^\gamma \mu P_{0,s+n\gamma}[\psi_{s+n \gamma} \times \cdot]ds  - \frac{1}{\gamma} \int_0^\gamma \beta_\gamma Q_{0,s}[\tilde{\psi}_{s} \times \cdot] ds  \right\|_{TV} \underset{n \to \infty}{\longrightarrow} 0.$$
By Cesaro's lemma, this allows us to conclude that, for any $\mu \in \cM_1(E_0)$ such that $\mu(\psi_0) < + \infty$,
\begin{multline*}
    \left\|\frac{1}{t} \int_0^t \mu P_{0,s} [\psi_{s} \times \cdot] ds -  \frac{1}{\gamma} \int_0^\gamma \beta_\gamma Q_{0,s} [\tilde{\psi}_{s} \times \cdot] ds \right\|_{TV} \\ \leq \frac{1}{\lfloor \frac{t}{\gamma} \rfloor} \sum_{k=0}^{\lfloor \frac{t}{\gamma} \rfloor}\left\|\frac{1}{\gamma}\int_0^\gamma \mu P_{0,s+k\gamma}[\psi_{s+k \gamma} \times \cdot]ds  - \frac{1}{\gamma} \int_0^\gamma \beta_\gamma Q_{0,s}[\tilde{\psi}_{s} \times \cdot] ds  \right\|_{TV} + \left\|\frac{1}{t} \int_{\lfloor \frac{t}{\gamma} \rfloor \gamma}^t \mu P_{0,s}[\psi_{s} \times \cdot] ds\right\|_{TV}
    \underset{t \to \infty}{\longrightarrow} 0,
\end{multline*}
which concludes the proof of \eqref{tv}. \medskip

\textit{Third step.} In the same manner, we now prove that, for any $\mu \in \cM_1(E_0)$ such that $\mu(\psi_0) < + \infty$,
\begin{equation}
    \label{final}
    \left\|\frac{1}{t} \int_0^t \mu P_{0,s} ds -  \frac{1}{\gamma} \int_0^\gamma \beta_\gamma Q_{0,s} ds \right\|_{TV} \underset{t \to \infty}{\longrightarrow} 0.
\end{equation}
In fact, for any function $f$ bounded by $1$ and $\mu \in \cM_1(E_0)$ such that $\mu(\psi_0) < + \infty$,
\begin{multline*}\left| \mu P_{0,s+n \gamma}\left[\psi_{s+n\gamma} \times \frac{f}{\psi_{s+n \gamma}}\right] - \beta_\gamma Q_{0,s}\left[\tilde{\psi}_s \times \frac{f}{\tilde{\psi}_s}\right]\right| \\\leq  \left| \mu P_{0,s+n \gamma}\left[\psi_{s+n\gamma} \times \frac{f}{\psi_{s+n \gamma}}\right] - \beta_\gamma Q_{0,s}\left[\tilde{\psi}_s \times \frac{f}{{\psi}_{s+n\gamma}}\right]\right| + \left| \beta_\gamma Q_{0,s}\left[\tilde{\psi}_s \times \frac{f}{{\psi}_{s+n\gamma}}\right] - \beta_\gamma Q_{0,s}\left[\tilde{\psi}_s \times \frac{f}{\tilde{\psi}_s}\right]\right| \\
\leq  \left\| \mu P_{0,s+n \gamma}[\psi_{s+n \gamma} \times \cdot] - \beta_\gamma Q_{0,s}[\tilde{\psi}_s \times \cdot] \right\|_{TV} + \left| \beta_\gamma Q_{0,s}\left[\tilde{\psi}_s \times \frac{f}{{\psi}_{s+n\gamma}}\right] - \beta_\gamma Q_{0,s}\left[\tilde{\psi}_s \times \frac{f}{\tilde{\psi}_s}\right]\right|.
\end{multline*}
We now remark that, since $\psi_{s+n \gamma} \geq 1$ for any $s$ and $n \in \Z_+$, one has that
$$\left| \frac{\tilde{\psi}_s}{\psi_{s+n \gamma}} - 1 \right| \leq 1 + \tilde{\psi}_s.$$
Since $(\psi_{s + n \gamma})_{n \in \Z_+}$ converges pointwisely towards $\tilde{\psi}_s$ and $\beta_\gamma Q_{0,s} \tilde{\psi}_s < + \infty$, Lebesgue's dominated convergence theorem implies
$$\sup_{f \in \cB_1(E)} \left| \beta_\gamma Q_{0,s}\left[\tilde{\psi}_s \times \frac{f}{{\psi}_{s+n\gamma}}\right] - \beta_\gamma Q_{0,s}\left[\tilde{\psi}_s \times \frac{f}{\tilde{\psi}_s}\right]\right| \underset{n \to \infty}{\longrightarrow} 0.$$
Then, using \eqref{jjj}, one has 
\begin{equation*}
    \left\| \mu P_{0,s+n \gamma} - \beta_\gamma Q_{0,s} \right\|_{TV} \underset{n \to \infty}{\longrightarrow} 0, 
\end{equation*}
which allows us to conclude \eqref{final}, using the same argument as in the first step. \medskip

\textit{Fourth step.} In order to show the $\L^2$-ergodic theorem, we let $f \in \cB(E)$. 
For any $x \in E_0$ and $t \geq 0$,
\begin{multline*}\E_{0,x}\left[\left|\frac{1}{t} \int_0^t f(X_s)ds - \E_{0,x}\left[\frac{1}{t} \int_0^t f(X_s)ds\right] \right|^2\right] \\= \frac{2}{t^2} \int_0^t \int_s^t \left(\E_{0,x}[f(X_s)f(X_u)] - \E_{0,x}[f(X_s)]\E_{0,x}[f(X_u)]\right)duds 
\\ = \frac{2}{t^2} \int_0^t \int_s^t \E_{0,x}\left[f(X_s)\left(f(X_u)  - \E_{0,x}[f(X_u)]\right)\right]duds
\\ = \frac{2}{t^2} \int_0^t \int_s^t \E_{0,x}\left[f(X_s)\left(\E_{s,X_s}[f(X_u)]  - \E_{s,\delta_x P_{0,s}}[f(X_u)]\right)\right]duds,\end{multline*}
whereby the Markov property was used on the last line. By \eqref{mixing} (weak ergodicity) and \eqref{maj}, one obtains for any $s \leq t$,
\begin{equation}
    \label{mixing-prop}
    \left|\E_{s,X_s}[f(X_t)]  - \E_{s,\delta_x P_{0,s}}[f(X_t)]\right| \leq C'' \|f\|_\infty [\psi_s(X_s) + \psi_0(x)] e^{- \kappa (t-s)},~~~~\P_{0,x}-\text{almost surely,}
\end{equation}
whereby $C''$ was defined in the first part.
As a result, for any $x \in E_0$ and $t \geq 0$, 
\begin{multline*}
    \E_{0,x}\left[\left|\frac{1}{t} \int_0^t f(X_s)ds - \E_{0,x}\left[\frac{1}{t} \int_0^t f(X_s)ds\right] \right|^2\right] \leq \frac{2C''\|f\|_\infty}{t^2} \int_0^t \int_s^t \E_{0,x}[|f(X_s)|(\psi_s(X_s)+\psi_0(x))] e^{-\kappa(u-s)}duds \\
    = \frac{2C''\|f\|_\infty}{t^2} \int_0^t \E_{0,x}[|f(X_s)|(\psi_s(X_s)+\psi_0(x))] e^{\kappa s} \frac{e^{-\kappa s} - e^{- \kappa t}}{\kappa}ds \\
     = \frac{2C''\|f\|_\infty}{\kappa t} \times \E_{0,x}\left[\frac{1}{t} \int_0^t |f(X_s)|(\psi_s(X_s) + \psi_0(x))ds\right] - \frac{2C''\|f\|_\infty e^{-\kappa t}}{\kappa t^2} \int_0^t e^{\kappa s} \E_{0,x}[|f(X_s)|(\psi_s(X_s)+\psi_0(x))]ds.
\end{multline*}
Then, by \eqref{tv}, there exists a constant $\Tilde{C} > 0$ such that, for any $x \in E_0$, when $t \to \infty$,
\begin{equation}
\label{variance}
\E_{0,x}\left[\left|\frac{1}{t} \int_0^t f(X_s)ds - \E_{0,x}\left[\frac{1}{t} \int_0^t f(X_s)ds\right] \right|^2\right] \leq \frac{\tilde{C} \|f\|_{\infty} \psi_0(x)}{t} \times \frac{1}{\gamma} \int_0^\gamma \beta_\gamma Q_{0,s}[|f|\tilde{\psi}_s] ds  + o\left(\frac{1}{t}\right).\end{equation}
Since $f \in \cB(E)$ and by definition of the total variation distance, \eqref{final} implies that, for all $x \in E_0$, 
$$\left| \frac{1}{t} \int_0^t P_{0,s}f(x) -   \frac{1}{\gamma} \int_0^\gamma \beta_\gamma Q_{0,s} f ds \right| \leq \|f\|_\infty  \left\|\frac{1}{t} \int_0^t \delta_x P_{0,s} ds -  \frac{1}{\gamma} \int_0^\gamma \beta_\gamma Q_{0,s} ds \right\|_{TV} \underset{t \to \infty}{\longrightarrow} 0.$$ 
Then, using \eqref{final}, one deduces that, for any $x \in E_0$ and bounded function $f$, 
\begin{multline*}\E_{0,x}\left[\left|\frac{1}{t} \int_0^t f(X_s)ds - \frac{1}{\gamma} \int_0^\gamma \beta_\gamma Q_{0,s} f ds \right|^2\right] \\ \leq 2 \left(\E_{0,x}\left[\left(\frac{1}{t} \int_0^t f(X_s)ds - \frac{1}{t} \int_0^t P_{0,s}f(x)\right)^2\right] + \left| \frac{1}{t} \int_0^t P_{0,s}f(x) -   \frac{1}{\gamma} \int_0^\gamma \beta_\gamma Q_{0,s} f ds \right|^2\right) \underset{t \to \infty}{\longrightarrow} 0.\end{multline*}
The convergence for any probability measure $\mu \in \cM_1(E_0)$ comes from Lebesgue's dominated convergence theorem.  \medskip

\textit{Fifth step.} We now fix non-negative $f \in \cB(E)$, and $\mu \in \cM_1(E_0)$ satisfying $\mu(\psi_0) < + \infty$. The following proof is inspired by the proof of \cite[Theorem 12]{Vass2018}.

Since $\mu(\psi_0) < + \infty$, the inequality \eqref{variance} entails that there exists a finite constant $C_{f,\mu} \in (0,\infty)$ such that, for $t$ large enough,
$$\E_{0,\mu}\left[\left|\frac{1}{t} \int_0^t f(X_s)ds - \E_{0,\mu}\left[\frac{1}{t} \int_0^t f(X_s)ds\right] \right|^2\right] \leq \frac{C_{f,\mu}}{t}.$$
Then, for $n$ large enough,
$$\E_{0,\mu}\left[\left|\frac{1}{n^2} \int_0^{n^2} f(X_s)ds - \E_{0,\mu}\left[\frac{1}{n^2} \int_0^{n^2} f(X_s)ds\right] \right|^2\right] \leq \frac{C_{f,\mu}}{n^2}.$$
Then, by Chebyshev's inequality and the Borel-Cantelli lemma, this last inequality implies that
$$\left|\frac{1}{n^2} \int_0^{n^2} f(X_s)ds - \E_{0,\mu}\left[\frac{1}{n^2} \int_0^{n^2} f(X_s)ds\right] \right| \underset{n \to \infty}{{\longrightarrow}} 0,~~~~\P_{0,\mu}-\text{almost surely.}$$
One thereby obtains by the convergence \eqref{final} that
\begin{equation}\label{asdiscrete}\frac{1}{n^2} \int_0^{n^2} f(X_s)ds  \underset{n \to \infty}{{\longrightarrow}} \frac{1}{\gamma} \int_0^\gamma \beta_\gamma Q_{0,s} f ds,~~~~\P_{0,\mu}-\text{almost surely.}\end{equation}
Since the nonnegativity of $f$ is assumed, this implies that for any $t > 0$ we have
$$\int_0^{\lfloor \sqrt{t} \rfloor^2} f(X_s)ds \leq \int_0^t f(X_s)ds \leq \int_0^{\lceil \sqrt{t} \rceil^2} f(X_s)ds.$$
These inequalities and \eqref{asdiscrete} then give that 
$$\frac{1}{t} \int_0^t f(X_s)ds ~\underset{t \to \infty}{{\longrightarrow}}~ \frac{1}{\gamma} \int_0^\gamma \beta_\gamma Q_{0,s} f ds,~~~~\P_{0,\mu}-\text{almost surely.}$$
In order to conclude that the result holds for any bounded measurable function $f$, it is enough to decompose $f = f_+ - f_-$ with $f_+ := f \lor 0$ and $f_- = (-f) \lor 0$ and apply the above convergence to $f_+$ and $f_-$. This concludes the proof of Theorem \ref{thm}.
\end{proof}

\begin{proof}[Proof of Corollary \ref{corollary}] 
We remark as in the previous proof that, if $\|f\|_\infty \leq 1$ and $\psi_s = \1$, an upper-bound for the inequality \eqref{variance} can be obtained, which does not depend on $f$ and $x$. Likewise, the convergence \eqref{jjj} holds uniformly in the initial measure due to \eqref{mixing-prop}.
\end{proof}


\begin{remark}
The proof of Theorem \ref{thm}, as written above, does not allow to deal with semigroups satisfying a Doeblin condition with time-dependent constant $c_s$, that is such that there exists $t_0 \geq 0$ and a family of probability measure $(\nu_t)_{t \geq 0}$ on $(E_t)_{t \geq 0}$ such that, for all $s \geq 0$ and $x \in E_s$,
$$\delta_x P_{s,s+t_0} \geq c_{s+t_0} \nu_{s+t_0}.$$
In fact, under the condition written above, we can show (see for example the proof of the formula (2.7) of \cite[Theorem 2.1]{CV2016}) that, for all $s \leq t$ and $\mu_1, \mu_2 \in \cM_1(E_s)$, 
$$\| \mu_1 P_{s,t} - \mu_2 P_{s,t} \|_{TV} \leq 2 \prod_{k=1}^{\left \lfloor \frac{t-s}{t_0} \right \rfloor} (1-c_{t-k t_0}).$$
Hence, by this last inequality with $\mu_1 = \delta_x P_{s,s+k\gamma}$, $\mu_2 = \delta_x$, replacing $s$ by $s+k\gamma$ and $t$ by $s+(k+n)\gamma$, one obtains
$$\|\delta_x P_{s,s+(k+n)\gamma} - \delta_x P_{s+k\gamma,s+(k+n)\gamma}\|_{TV} \leq 2 \prod_{l=1}^{\lfloor \frac{n\gamma}{t_0}\rfloor} (1-c_{s+(k+n)\gamma-lt_0}),$$
replacing therefore the inequality \eqref{premiere-etape} in the proof of Theorem \ref{thm}.
Plugging therefore this last inequality into the formula \eqref{pour-la-remark}, one obtains
$$\| \delta_{x} P_{s,s+(k+n)\gamma} - {\beta}_{\gamma} Q_{0,s} \|_{TV} \leq  2 \prod_{l=1}^{\left \lfloor \frac{n \gamma}{t_0} \right \rfloor} (1-c_{s + (k+n)\gamma-l t_0})  + \| \delta_{x} P_{s+k\gamma, s+(k+n)\gamma} - \beta_\gamma Q_{0,s} \|_{TV}.$$
Hence, we see that we cannot conclude a similar result when $c_s \longrightarrow 0$, as $s \to + \infty$ since, for $n$ fixed,
$$\limsup_{k \to \infty} \prod_{l=1}^{\left \lfloor \frac{n \gamma}{t_0} \right \rfloor} (1-c_{s+(k+n)\gamma-l t_0}) =  1.$$
\end{remark}

\section{Application to quasi-stationarity with moving boundaries}
\label{3}

In this section, $(X_t)_{t \geq 0}$ is assumed as being a time-homogeneous Markov process. 
We consider a family of measurable subsets $(A_t)_{t \geq 0}$ of $E$, and denote the hitting time
$$\tau_A := \inf\{t \geq 0 : X_t \in A_t\}.$$
For all $s \leq t$, denote by $\cF_{s,t}$ the $\sigma$-field generated by the family $(X_u)_{s \leq u \leq t}$ and $\cF_t := \cF_{0,t}$. 
Assume that $\tau_A$ is a stopping time with respect to the filtration $(\cF_{t})_{t \geq 0}$. Assume also that for any $x \not \in A_0$, 
$$\P_{0,x}[\tau_A < + \infty] = 1 ~~~~~~\text{ and }~~~~~~ \P_{0,x}[\tau_A > t] > 0,~~\forall t \geq 0.$$
We will be interested in a notion of \textit{quasi-stationarity with moving boundaries}, which studies the asymptotic behavior of the Markov process $(X_t)_{t \geq 0}$ conditioned not to hit $(A_t)_{t \geq 0}$ up to the time $t$. For non-moving boundaries ($A_t = A_0$ for any $t \geq 0$), the \textit{quasi-limiting distribution}  is defined as a probability measure $\alpha$ such that, for at least one initial measure $\mu$ and for all measurable subsets $\mathcal{A} \subset E$,
$$\P_{0,\mu}[X_t \in \mathcal{A} | \tau_A > t] \underset{t \to \infty}{\longrightarrow} \alpha(\mathcal{A}).$$
Such a definition is equivalent (still in the non-moving framework) to the notion of \textit{quasi-stationary distribution} defined as a probability measure $\alpha$ such that, for any $t \geq 0$, 
\begin{equation}
\label{qsd-equation}\P_{0,\alpha}[X_t \in \cdot | \tau_A > t] = \alpha.\end{equation}
If quasi-limiting and quasi-stationary distributions are in general well-defined for time-homogeneous Markov processes and non-moving boundaries (see \cite{CMSM,MV2012} for a general overview on the theory of quasi-stationarity), these notions could be not well-defined for time-inhomogeneous Markov processes or moving boundaries and are not equivalent anymore. In particular, under reasonable assumptions on irreducibility, it was shown in \cite{OCAFRAIN2017} that the notion of quasi-stationary distribution as defined by \eqref{qsd-equation} is not well-defined for time-homogeneous Markov processes absorbed by moving boundaries. 

Another asymptotic notion to study is the \textit{quasi-ergodic distribution}, related to a conditional version of the ergodic theorem and usually defined as follows.
\begin{definition}
\label{qed}
 A probability measure $\beta$ is a \textit{quasi-ergodic distribution} if, for some initial measure $\mu \in \cM_1(E \setminus A_0)$ and for any bounded continuous function $f$,
$$\E_{0,\mu}\left[\frac{1}{t} \int_0^t f(X_s)ds \middle| \tau_A > t\right] \underset{t \to \infty}{\longrightarrow} \beta(f).$$
\end{definition}
 In the time-homogeneous setting (in particular for non-moving boundaries), this notion has been extensively studied (see for example \cite{BR1999,CV2017,CJ2018,colonius2021quasi,Darroch1965,Guoman2018,HYZ2019,GHY2016,OCAFRAIN2017}). In the \textit{"moving boundaries"} framework, the existence of quasi-ergodic distributions has been dealt with in \cite{OCAFRAIN2017} for Markov chains on finite state spaces absorbed by periodic boundaries, and in \cite{ocafrain2018} for processes satisfying a Champagnat-Villemonais condition (see Assumption (A') set later) absorbed by converging or periodic boundaries. In this last paper, the existence of the quasi-ergodic distribution is dealt with through the following inequality  (see \cite[Theorem 1]{ocafrain2018}), holding for any initial state $x$, $s \leq t$ and for some constant $C, \gamma > 0$ independent of $x$, $s$ and $t$:
$$\|\P_{0,x}(X_s \in \cdot | \tau_A > t) - \Q_{0,x}(X_s \in \cdot)\|_{TV} \leq C e^{-\gamma (t-s)},$$
where the family of probability measure $(\Q_{s,x})_{s \geq 0, x \in E_s}$ is defined by
$$\Q_{s,x}[\Gamma] := \lim_{T \to \infty} \P_{s,x}[\Gamma | \tau_A > T],~~~~\forall s \leq t, x \in E \setminus A_s, \Gamma \in \cF_{s,t}.$$
Moreover, \cite[Proposition 3.1]{CV2016}, there exists a family of positive bounded functions $(\eta_t)_{t \geq 0}$ defined in  such that, for all $s \leq t$ and $x \in E_s$,
$$\E_{s,x}(\eta_t(X_t)\1_{\tau_A > t}) = \eta_s(x).$$
Then, we can show (this is actually shown in \cite{CV2016}) that
$$\Q_{s,x}(\Gamma) = \E_{s,x}(\1_{\Gamma, \tau_A > t} \frac{\eta_t(X_t)}{\eta_s(x)})$$
and that, for all $\mu \in \cM_1(E_0)$,
$$\|\P_{0,\mu}(X_s \in \cdot | \tau_A > t) - \Q_{0,\eta_0*\mu}(X_s \in \cdot)\|_{TV} \leq C e^{-\gamma (t-s)}.$$
where $\eta_0 * \mu(dx) := \frac{\eta_0(x)\mu(dx)}{\mu(\eta_0)}$. By triangular inequality, one has
\begin{equation}
\label{using}
\left\|\frac{1}{t} \int_0^t \P_{0,\mu}[X_s \in \cdot | \tau_A > t]ds - \frac{1}{t} \int_0^t \Q_{0,\eta_0 * \mu}[X_s \in \cdot]ds \right\|_{TV} \leq \frac{C}{\gamma t},~~~~\forall t > 0,\end{equation}
In particular, the inequality \eqref{using} implies that there exists a quasi-ergodic distribution $\beta$ for the process $(X_t)_{t \geq 0}$ absorbed by $(A_t)_{t \geq 0}$ if and only if there exist some probability measures $\mu \in \cM_1(E_0)$ such that $\frac{1}{t} \int_0^t \Q_{0,\eta_0 * \mu}[X_s \in \cdot]ds$ converges weakly to $\beta$, when $t$ goes to infinity. In other terms, under Assumption (A'), the existence of a quasi-ergodic distribution for the absorbed process is equivalent to the law of large number for its $Q$-process.



Assumption (A') is now set. 

\begin{assumptiona}
\label{doeblin-assumption} There exists a family of probability measures $(\nu_t)_{t \geq 0}$, defined on $E \setminus A_t$ for each $t$, such that:
\begin{enumerate}[({A'}1)]
\item There exists $t_0 \geq 0$ and $c_1 > 0$ such that
$$\P_{s,x}[X_{s+t_0} \in \cdot | \tau_A > s + t_0] \geq c_1 \nu_{s+t_0},~~~~\forall s \geq 0, \forall x \in E \setminus A_s.$$
\item There exists $c_2 > 0$ such that
$$\P_{s,\nu_s}[\tau_A > t] \geq c_2 \P_{s,x}[\tau_A > t],~~~~\forall s \leq t, \forall x \in E \setminus A_s.$$
\end{enumerate} 
\end{assumptiona}
In what follows, we say that the couple $\{(X_t)_{t \geq 0}, (A_t)_{t \geq 0}\}$ satisfies Assumption (A') when it holds for the Markov process $(X_t)_{t \geq 0}$ considered as absorbed by the moving boundary $(A_t)_{t \geq 0}$. 

The condition (A'1) is a conditional version of the Doeblin condition \eqref{doeblin} and (A'2) is a Harnack-like inequality on the probabilities of surviving, necessary to deal with the conditioning. They are equivalent to the set of conditions presented in \cite[Definition 2.2]{BCG2017}, when the non-conservative semigroup is sub-Markovian. In the time-homogeneous framework, we obtain the Champagnat-Villemonais condition defined in \cite{CV2014} (see Assumption (A)), shown as being equivalent to the exponential uniform convergence to quasi-stationarity in total variation. 

\textcolor{black}{In \cite{ocafrain2018}, the existence of a unique quasi-ergodic distribution has been only proved for converging or periodic boundaries. However, we can expect such a result of existence (and uniqueness) for other kinds of movement for the boundary. Hence, the aim of this section is to extend the results on the existence of quasi-ergodic distributions obtained in \cite{ocafrain2018} for Markov processes absorbed by asymptotically periodic moving boundaries.}

 Now, let us state the following theorem.


\begin{theorem}
\label{qsd}
Assume that there exists a $\gamma$-periodic sequence of subsets $(B_t)_{t \geq 0}$ such that, for any $s \in [0,\gamma)$, $$E_s' := E \setminus \bigcap_{k \in \Z_+} \bigcup_{l \geq k} A_{s+l \gamma} \cup B_s \ne \emptyset$$
and there exists $x_s \in E_s$ such that, for any $n \leq N$,
\begin{equation}
    \label{condition}
    \|\P_{s+k \gamma, x_s}[X_{s+(k+n)\gamma} \in \cdot, \tau_A > s+(k+N)\gamma] - \P_{s, x_s}[X_{s+n\gamma} \in \cdot, \tau_B > s+N\gamma]\|_{TV}  \underset{k \to \infty}{\longrightarrow} 0.
\end{equation}
Assume also that Assumption (A') is satisfied by the couples $\{(X_t)_{t \geq 0}, (A_t)_{t \geq 0}\}$ and $\{(X_t)_{t \geq 0}, (B_t)_{t \geq 0}\}$. 

Then there exists a probability measure $\beta \in \cM_1(E)$ such that 
\begin{equation}
\label{uniform}
\sup_{\mu \in \cM_1(E \setminus A_0)} \sup_{f \in \cB_1(E)} \E_{0,\mu}\left[ \left| \frac{1}{t} \int_0^t f(X_s)ds - \beta(f) \right|^2 \middle| \tau_A > t\right] \underset{t \to \infty}{\longrightarrow} 0.\end{equation}
\end{theorem}

\begin{remark}
Remark that the condition \eqref{condition} implies that, for any $n \in \Z_+$,
$$\P_{s+k \gamma,x_s}[\tau_A > s +(k+n) \gamma] \underset{k \to \infty}{\longrightarrow} \P_{s,x_s}[\tau_B > s + n \gamma].$$
Under the additional condition $B_t \subset A_t$ for all $t \geq 0$, these two conditions are equivalent, since for all $n \leq N$, 
\begin{multline*}
\|\P_{s+k \gamma, x_s}[X_{s+(k+n)\gamma} \in \cdot, \tau_A > s+(k+N)\gamma] - \P_{s, x_s}[X_{s+n\gamma} \in \cdot, \tau_B > s+N\gamma]\|_{TV} \\ = \|\P_{s+k\gamma,x_s}[X_{s+(k+n)\gamma} \in \cdot, \tau_B \leq s+(k+N) \gamma  < \tau_A] \|_{TV} \\ \leq \P_{s+k\gamma,x_s}[\tau_B \leq s+(k+N) \gamma  < \tau_A]   \\=  |\P_{s+k \gamma, x_s}[\tau_A > s+(k+N)\gamma] - \P_{s, x_s}[\tau_B > s+N\gamma]|,\end{multline*}
where we used the periodicity of $(B_t)_{t \geq 0}$, writing that $\P_{s, x_s}[X_{s+n\gamma} \in \cdot, \tau_B > s+N\gamma] = \P_{s+k\gamma, x_s}[X_{s+(k+n)\gamma} \in \cdot, \tau_B > s+(k+N)\gamma]$ for all $k \in \Z_+$. 
This implies the following corollary.
\end{remark}

\begin{corollary}
\label{qsd-cor}
Assume that there exists a $\gamma$-periodic sequence of subsets $(B_t)_{t \geq 0}$, with $B_t \subset A_t$ for all $t \geq 0$, such that, for any $s \in [0,\gamma)$, there exists $x_s \in E'_s$ such that, for any $n \leq N$,
$$\P_{s+k \gamma,x_s}[\tau_A > s +(k+n) \gamma] \underset{k \to \infty}{\longrightarrow} \P_{s,x_s}[\tau_B > s + n \gamma].$$
Assume also that Assumption (A') is satisfied by $\{(X_t)_{t \geq 0}, (A_t)_{t \geq 0}\}$ and $\{(X_t)_{t \geq 0}, (B_t)_{t \geq 0}\}$. 

Then there exists $\beta \in \cM_1(E)$ such that \eqref{uniform} holds.

\end{corollary}


\begin{proof}[Proof of Theorem \ref{qsd}]
Since $\{(X_t)_{t \geq 0}, (B_t)_{t \geq 0}\}$ satisfies Assumption (A') and $(B_t)_{t \geq 0}$ is a periodic boundary, we already know by \cite[Theorem 2]{ocafrain2018} that, for any initial distribution $\mu$, $t \mapsto \frac{1}{t} \int_0^t \P_{0,\mu}[X_s \in \cdot | \tau_B > t]ds$ converges weakly to a quasi-ergodic distribution $\beta$. 

The main idea of this proof is to apply Corollary \ref{corollary}. Since $\{(X_t)_{t \geq 0}, (A_t)_{t \geq 0}\}$ and $\{(X_t)_{t \geq 0}, (B_t)_{t \geq 0}\}$ satisfy Assumption (A'), \cite[Theorem 1]{ocafrain2018} implies that there exist two families of probability measures $(\Q^A_{s,x})_{s \geq 0, x \in E \setminus A_s}$ and $(\Q^B_{s,x})_{s \geq 0, x \in E \setminus B_s}$ such that, for any $s \leq t$, $x \in E \setminus A_s$, $y \in E \setminus B_s$ and $\Gamma \in \cF_{s,t}$,
$$\Q_{s,x}^A[\Gamma] = \lim_{T \to \infty} \P_{s,x}[\Gamma | \tau_A > T],~~\text{ and } ~~\Q_{s,y}^B[\Gamma] = \lim_{T \to \infty} \P_{s,y}[\Gamma | \tau_B > T].$$
In particular, the quasi-ergodic distribution $\beta$ is the limit of $t \mapsto \frac{1}{t} \int_0^t \Q^B_{0, \mu}[X_s \in \cdot]ds$, when $t$ goes to infinity (see \cite[Theorem 5]{ocafrain2018}).
Also, by \cite[Theorem 1]{ocafrain2018}, there exists a constant $C > 0$ and $\kappa > 0$ such that, for any $s \leq t \leq T$, for any $x \in E \setminus A_{s}$, 
$$\| \Q^A_{s,x}[X_t \in \cdot] - \P_{s, x}[X_{t} \in \cdot | \tau_A > T] \|_{TV} \leq C e^{-\kappa (T-t)},$$
and for any $x \in E \setminus B_s$, $$\| \Q^B_{s,x}[X_{t} \in \cdot] - \P_{s, x}[X_{t} \in \cdot | \tau_B > T] \|_{TV} \leq C e^{-\kappa (T-t)}.$$
Moreover, for any $s \leq t \leq T$ and  $x \in E'_s$,
\begin{align}
\| \P_{s,x}[X_{t} \in \cdot | \tau_A > T] - \P_{s,x}[X_{t} &\in \cdot | \tau_B > T] \|_{TV} \notag \\&
=\left|\left| \frac{\P_{s,x}[X_{t} \in \cdot, \tau_A > T]}{\P_{s,x}[\tau_A > T]} - \frac{\P_{s,x}[X_{t} \in \cdot, \tau_B > T]}{\P_{s,x}[\tau_B > T]} \right| \right|_{TV} \notag \\ &= \left|\left|  \frac{\P_{s,x}(\tau_B > T)}{\P_{s,x}(\tau_A > T)}\frac{\P_{s,x}[X_{t} \in \cdot, \tau_A > T]}{\P_{s,x}[\tau_B > T]} - \frac{\P_{s,x}[X_{t} \in \cdot, \tau_B > T]}{\P_{s,x}[\tau_B > T]} \right| \right|_{TV} \notag \\
&\leq  \left|\left| \frac{\P_{s,x}(\tau_B > T)}{\P_{s,x}(\tau_A > T)}\frac{\P_{s,x}[X_{t} \in \cdot, \tau_A > T]}{\P_{s,x}[\tau_B > T]} - \frac{\P_{s,x}[X_{t} \in \cdot, \tau_A > T]}{\P_{s,x}[\tau_B > T]} \right| \right|_{TV} \notag \\&~~~~~~+ \left|\left|  \frac{\P_{s,x}[X_{t} \in \cdot, \tau_A > T]}{\P_{s,x}[\tau_B > T]} - \frac{\P_{s,x}[X_{t} \in \cdot, \tau_B > T]}{\P_{s,x}[\tau_B > T]}  \right| \right|_{TV} \notag \\
&\leq \frac{|\P_{s,x}(\tau_B > T) - \P_{s,x}(\tau_A > T)|}{\P_{s,x}(\tau_B > T)} + \frac{\|\P_{s,x}[X_t \in \cdot, \tau_A > T] - \P_{s,x}[X_t \in \cdot, \tau_B > T]\|_{TV}}{\P_{s,x}[\tau_B > T]} \notag \\
&\leq 2 \frac{\|\P_{s,x}[X_t \in \cdot, \tau_A > T] - \P_{s,x}[X_t \in \cdot, \tau_B > T]\|_{TV}}{\P_{s,x}[\tau_B > T]},
\label{degueu}
\end{align}
since $|\P_{s,x}(\tau_B > T) - \P_{s,x}(\tau_A > T)| \leq \|\P_{s,x}[X_t \in \cdot, \tau_A > T] - \P_{s,x}[X_t \in \cdot, \tau_B > T]\|_{TV}$.
Then, we obtain for any $s \leq t \leq T$ and $x \in E'_s$, 
\begin{equation}
    \|\Q_{s,x}^A[X_{t} \in \cdot] - \Q_{s,x}^B[X_{t} \in \cdot] \|_{TV} \leq 2 C e^{-\kappa (T - t)} + 2 \frac{\|\P_{s,x}[X_t \in \cdot, \tau_A > T] - \P_{s,x}[X_t \in \cdot, \tau_B > T]\|_{TV}}{\P_{s,x}[\tau_B > T]}.
\label{control}
\end{equation}
The condition \eqref{condition} implies the existence of $x_s \in E_s$ such that, for any $n \leq N$, for all $k \in \Z_+$,
$$\lim_{k \to \infty} \|\P_{s+k \gamma, x_s}[X_{s+(k+n)\gamma} \in \cdot, \tau_A > s+(k+N)\gamma] - \P_{s, x_s}[X_{s+n\gamma} \in \cdot, \tau_B > s+N\gamma]\|_{TV} = 0,$$
which implies by \eqref{control} that, for any $n \leq N$,
$$\limsup_{k \to \infty}  \|\Q_{s+k\gamma,x_s}^A[X_{s+(k+n)\gamma} \in \cdot] - \Q_{s+k\gamma,x_s}^B[X_{s+(k+n)\gamma} \in \cdot] \|_{TV} \leq 2Ce^{-\kappa \gamma (N-n)}.$$
Now, letting $N \to \infty$, for any $n \in \Z_+$
$$\lim_{k \to \infty}  \|\Q_{s+k\gamma,x_s}^A[X_{s+(k+n)\gamma} \in \cdot] - \Q_{s+k\gamma,x_s}^B[X_{s+(k+n)\gamma} \in \cdot] \|_{TV} = \lim_{k \to \infty} \| \Q_{s+k\gamma,x_s}^A(X_{s+(k+n)\gamma}\in \cdot)-\Q_{s,x_s}^B(X_{s+n\gamma}\in \cdot)\|_{TV}= 0.$$
In other words, the semigroup $(Q^A_{s,t})_{s \leq t}$ defined by 
$$Q_{s,t}^Af(x) := \E_{s,x}^{\Q^A}(f(X_t)),~~~~\forall s \leq t, \forall f \in \cB(E \setminus A_t), \forall x \in E \setminus A_s,$$
 is asymptotically periodic (according to the Definition \ref{asymptotic-periodicity}, with $\psi_s = \tilde{\psi}_s = 1$ for all $s \geq 0$), associated to the auxiliary semigroup $(Q^B_{s,t})_{s \leq t}$ defined by 
$$Q_{s,t}^Bf(x) := \E_{s,x}^{\Q^B}(f(X_t)),~~~~\forall s \leq t, \forall f \in \cB(E \setminus B_t), \forall x \in E \setminus B_s.$$ Moreover, assumptions (A') satisfied for $\{(X_t)_{t \geq 0}, (A_t)_{t \geq 0}\}$ and $\{(X_t)_{t \geq 0}, (B_t)_{t \geq 0}\}$ imply that Doeblin condition  holds for these two $Q$-processes. As a matter of fact, by the Markov property, for all $s \leq t \leq T$ and $x \in E \setminus A_s$,
\begin{align}
 \P_{s,x}(X_t \in \cdot | \tau_A > T) &= \E_{s,x}\left[\1_{X_t \in \cdot, \tau_A > t} \frac{\P_{t,X_t}(\tau_A > T)}{\P_{s,\mu}(\tau_A > T)}\right] \notag \\
&= \E_{s,x}\left[\frac{\1_{X_t \in \cdot, \tau_A > t}}{\P_{s,x}(\tau_A > t)} \frac{\P_{t,X_t}(\tau_A > T)}{\P_{t,\phi_{t,s}(\mu)}(\tau_A > T)}\right] \notag \\
&= \E_{s,x}\left[\1_{X_t \in \cdot} \frac{\P_{t,X_t}(\tau_A > T)}{\P_{t,\phi_{t,s}(\delta_x)}(\tau_A > T)} \middle| \tau_A > t\right]. \label{pour-doeblin}\end{align}
By $(A'_1)$, for all $s \geq 0$, $T \geq s+t_0$ and $x \in E \setminus A_s$,
$$\E_{s,x}\left[\1_{X_{s+t_0} \in \cdot} \frac{\P_{s+t_0,X_{s+t_0}}(\tau_A > T)}{\P_{s+t_0,\phi_{s+t_0,s}(\delta_x)}(\tau_A > T)} \middle| \tau_A > s+t_0\right] \geq c_1 \int_{\cdot} \nu_{s+t_0}(dy) \frac{\P_{s+t_0,y}(\tau_A > T)}{\P_{s+t_0,\phi_{s+t_0,s}(\delta_x)}(\tau_A > T)},$$
that is to say by \eqref{pour-doeblin} that
$$\P_{s,x}(X_{s+t_0} \in \cdot | \tau_A > T) \geq  c_1 \int_{\cdot} \nu_{s+t_0}(dy) \frac{\P_{s+t_0,y}(\tau_A > T)}{\P_{s+t_0,\phi_{s+t_0,s}(\delta_x)}(\tau_A > T)}.$$
Letting $T \to \infty$ in this last inequality and using \cite[Proposition 3.1]{CV2016}, for all $s \geq 0$ and $x \in E \setminus A_s$,
$$\Q_{s,x}^A(X_{s+t_0} \in \cdot) \geq c_1 \int_\cdot \nu_{s+t_0}(dy) \frac{\eta_{s+t_0}(y)}{\phi_{s+t_0,s}(\delta_x)(\eta_{s+t_0})}.$$ 
The measure $ \int_\cdot \nu_{s+t_0}(dy) \frac{\eta_{s+t_0}(y)}{\phi_{s+t_0,s}(\delta_x)(\eta_{s+t_0})}$ is then a positive measure whose the mass is lower-bounded by $c_2$ by $(A_2')$, since for all $s \geq 0$ and $T \geq s + t_0$,
$$\int_{E \setminus A_{s+t_0}} \nu_{s+t_0} \frac{\P_{t,x}(\tau_A > T)}{\P_{t,\phi_{t,s}(\delta_x)}(\tau_A > T)} \geq c_2.$$
This proves therefore a Doeblin condition for the semigroup $(Q_{s,t}^A)_{s \leq t}$. The same reasoning applies to show also a Doeblin condition for the semigroup $(Q_{s,t}^B)_{s \leq t}$. 
  Then, using \eqref{using} then Corollary \ref{corollary},\begin{equation*}
\lim_{t \to \infty} \frac{1}{t} \int_0^t \P_{0,\mu}[X_s \in \cdot | \tau_A > t]ds = \lim_{t \to \infty} \frac{1}{t} \int_0^t \Q_{0,\eta_0*\mu}^A(X_s \in \cdot)ds = \lim_{t \to \infty} \frac{1}{t} \int_0^t \Q^B_{0,\eta_0*\mu}[X_s \in \cdot]ds = \beta,\end{equation*}
where the limits refer to the convergence in total variation and hold uniformly in the initial measure.

For any $\mu \in \cM_1(E \setminus A_0)$, $f \in \cB_1(E)$ and $t \geq 0$, 
$$\E_{0,\mu}\left[\left| \frac{1}{t} \int_0^t f(X_s)ds \right|^2\middle| \tau_A > t\right] = \frac{2}{t^2} \int_0^t \int_s^t \E_{0,\mu}[f(X_s)f(X_u) | \tau_A > t]duds.$$
Then, by \cite[Theorem 1]{ocafrain2018}, for any $s \leq u \leq t$, for any $\mu \in \cM_1(E \setminus A_0)$ and $f \in \cB(E)$,
$$\left|  \E_{0,\mu}[f(X_s)f(X_u) | \tau_A > t] - \E^{\Q^A}_{0,\eta_0*\mu}[f(X_s)f(X_u)]\right| \leq C \|f\|_\infty e^{-\kappa (t-u)},$$
where the expectation $\E^{\Q^A}_{0,\eta_0*\mu}$ is associated to the probability measure $\Q_{0,\eta_0*\mu}^A$. 
Hence, for any $\mu \in \cM_1(E \setminus A_0)$, $f \in \cB_1(E)$ and $t > 0$,
\begin{align*}
    \left| \E_{0,\mu}\left[\left| \frac{1}{t} \int_0^t f(X_s)ds - \beta(f) \right|^2\middle| \tau_A > t\right] - \E_{0,\eta_0*\mu}^{\Q^A}\left[\left| \frac{1}{t} \int_0^t f(X_s)ds - \beta(f) \right|^2\right] \right| &\leq \frac{4C}{t^2} \int_0^t \int_s^t e^{-\kappa(t-u)}duds \\
    &\leq \frac{4C}{\kappa t} - \frac{4C (1 - e^{-\kappa t})}{\kappa^2 t^2}.
\end{align*}
Moreover, since $(Q^A_{s,t})_{s \leq t}$ is asymptotically periodic in total variation and satisfies the Doeblin condition, like $(Q^B_{s,t})_{s \leq t}$, Corollary \ref{corollary} implies that
$$\sup_{\mu \in \cM_1(E \setminus A_0)} \sup_{f \in \cB_1(E)} \E_{0,\eta_0*\mu}^{\Q^A}\left[\left| \frac{1}{t} \int_0^t f(X_s)ds - \beta(f) \right|^2\right] \underset{t \to \infty}{\longrightarrow} 0.$$
Then
$$\sup_{\mu \in \cM_1(E \setminus A_0)} \sup_{f \in \cB_1(E)} \E_{0,\mu}\left[\left| \frac{1}{t} \int_0^t f(X_s)ds - \beta(f) \right|^2\middle| \tau_A > t\right] \underset{t \to \infty}{\longrightarrow} 0.$$
\end{proof}

\begin{remark}
It seems that Assumption (A') can be weaken by a conditional version of Assumption \ref{lyap-def}. In particular, such conditions can be derived from the Assumption (F) presented in \cite{CV2017c}, as it will be shown later by the preprint \cite{COV} in preparation.   
\end{remark}

\section{Examples}

\subsection{Asymptotically periodic Ornstein-Uhlenbeck processes}

Let $(X_t)_{t \geq 0}$ be a time-inhomogeneous diffusion process on $\R$ satisfying the following stochastic differential equation
$$dX_t = dW_t - \lambda(t) X_t dt,$$
where $(W_t)_{t \geq 0}$ is a one-dimensional Brownian motion and $\lambda : [0, \infty) \to [0, \infty)$ is a function such that  $$\sup_{t \geq 0} |\lambda(t)| < + \infty$$ and such that there exists $\gamma > 0$ such that $$\inf_{s \geq 0} \int_s^{s+\gamma} \lambda(u)du > 0.$$ By Itô's lemma, for any $s \leq t$,
$$X_t = e^{-\int_s^t \lambda(u)du}\left[X_s + \int_s^t e^{\int_s^u \lambda(v)dv} dW_u\right].$$ 
In particular, denoting $(P_{s,t})_{s \leq t}$ the semigroup associated to $(X_t)_{t \geq 0}$, for any $f \in \cB(\R)$, $t \geq 0$ and $x \in \R$,
$$P_{s,t}f(x) = \E\left[f\left(e^{-\int_s^t \lambda(u)du}x + e^{-\int_s^t \lambda(u)du} \sqrt{\int_s^t e^{2\int_s^u \lambda(v)dv}du} \times \cN(0,1)\right)\right],$$
where $\cN(0,1)$ denotes a standard Gaussian variable. 

\begin{theorem}
Assume that there exists a $\gamma$-periodic function $g$, bounded on $\R$, such that $\lambda \sim_{t \to \infty} g$. Then the assumptions of Theorem \ref{thm} hold.
\end{theorem}

\begin{proof}
In our case, the auxiliary semigroup $(Q_{s,t})_{s \leq t}$ of Definition \ref{asymptotic-periodicity} will be defined as follows: 
 for any $f \in \cB(\R)$, $t \geq 0$ and $x \in \R$,
$$Q_{s,t}f(x) = \E\left[f\left(e^{-\int_s^t g(u)du}x + e^{-\int_s^t g(u)du} \sqrt{\int_s^t e^{2\int_s^u g(v)dv}du} \times \cN(0,1)\right)\right].$$
In particular, the semigroup $(Q_{s,t})_{s \leq t}$ is associated to the process $(Y_t)_{t \geq 0}$ following 
$$dY_t = dW_t - g(t) Y_t dt.$$
We first remark that the function $\psi : x \mapsto 1 + x^2$ is a Lyapunov function for $(P_{s,t})_{s \leq t}$ and $(Q_{s,t})_{s \leq t}$. In fact, for any $s \geq 0$ and $x \in \R$,
\begin{align*}
    P_{s,s+\gamma} \psi(x) &= 1 + e^{-2 \int_s^{s+\gamma} \lambda(u)du}x^2 + e^{-2\int_s^{s+\gamma} \lambda(u)du} \int_s^{s+\gamma} e^{2\int_s^u \lambda(v)dv}du  \\
    &= e^{-2\int_s^{s+\gamma} \lambda(u)du} \psi(x) + 1 - e^{-2\int_s^{s+\gamma} \lambda(u)du} +  e^{-2\int_s^{s+\gamma} \lambda(u)du} \int_s^{s+\gamma} e^{2\int_s^u \lambda(v)dv}du \\
    &\leq e^{-2 \gamma c_{\inf}} \psi(x) + C, 
\end{align*}
where $C \in (0,+\infty)$ and $c_{\inf} := \inf_{t \geq 0} \frac{1}{\gamma} \int_t^{t+\gamma}\lambda(u)du > 0$. 
Taking $\theta \in (e^{-2 \gamma c_{\inf}},1)$, there exists a compact set $K$ such that, for any $s \geq 0$,
$$P_{s,s+\gamma} \psi(x) \leq \theta \psi(x) + C \1_K(x).$$ Moreover, for any $s \geq 0$ and $t \in [0, \gamma)$, the function $P_{s,s+t}\psi/\psi$ is upper-bounded uniformly in $s$ and $t$. It remains therefore to prove Assumption \ref{lyap-def} (i) for $(P_{s,t})_{s \leq t}$, which is a consequence of the following lemma.
\begin{lemma}
\label{gaussian}
For any $a,b_{-},b_{+} > 0$, define the subset $\cC(a,b_{-},b_{+}) \subset \cM_1(\R)$ as 
$$\cC(a,b_{-},b_{+}) := \{ \cN(m,\sigma) : m \in [-a,a], \sigma \in [b_{-},b_{+}]\}.$$
Then, for any $a,b_{-},b_{+} > 0$, there exists a probability measure $\nu$ and a constant $c > 0$ such that, for any $\mu \in \cC(a,b_{-},b_{+})$,
$$\mu \geq c \nu.$$
\end{lemma}
The proof of this lemma is postponed after the end of this proof. 

Since $\lambda \sim_{t \to \infty} g$ and these two functions are bounded on $\R_+$, Lebesgue's dominated convergence theorem implies that, for all $s \leq t$, 
$$\left| \int_{s+k\gamma}^{t+k\gamma} \lambda(u)du - \int_s^{t} g(u)du\right| \underset{k \to \infty}{\longrightarrow} 0.$$
In the same way, for all $s \leq t$, 
$$\int_{s+k\gamma}^{t+k\gamma} e^{2\int_{s+k\gamma}^u \lambda(v)dv}du \underset{k \to \infty}{\longrightarrow} \int_s^{t} e^{2\int_s^u g(v)dv}du.$$
Hence, for any $s \leq t$, 
$$e^{-\int_{s+k\gamma}^{t+k\gamma} \lambda(u)du} \underset{k \to \infty}{\longrightarrow} e^{-\int_s^{t} g(u)du},$$
and $$e^{-\int_{s+k \gamma}^{t+k\gamma} \lambda(u)du} \sqrt{\int_{s+k\gamma}^{t+k\gamma} e^{2\int_{s+k\gamma}^u \lambda(v)dv}du} \underset{k \to \infty}{\longrightarrow}  e^{-\int_s^{t} g(u)du} \sqrt{\int_s^{t} e^{2\int_s^u g(v)dv}du}.$$
Using \cite[Theorem 1.3.]{devroye2018total}, for any $x \in \R$, 
\begin{equation}
\label{convergence-tv}
\| \delta_x P_{s+k\gamma, t+k\gamma} - \delta_x Q_{s+k\gamma, t+k\gamma} \|_{TV} \underset{k \to \infty}{\longrightarrow} 0.\end{equation}
To deduce the convergence in $\psi$-distance, we will inspire from the proof of \cite[Lemma 3.1]{hening2018stochastic}.
Since the variances are uniformly bounded in $k$ (for $s \leq t$ fixed), there exists $H > 0$ such that, for any $k \in \N$ and $s \leq t$, 
\begin{equation}
\label{H}
\delta_x P_{s+k\gamma, t+k\gamma}[\psi^2] \leq H~~\text{ and }~~\delta_x Q_{s, t}[\psi^2] \leq H.\end{equation}
Since $\lim_{|x| \to \infty} \frac{\psi(x)}{\psi^2(x)} = 0$, for any $\epsilon > 0$, there exists $l_\epsilon > 0$ such that, for any function $f$ such that $|f| \leq \psi$ and for any $|x| \geq l_\epsilon$,
$$|f(x)| \leq \frac{\epsilon \psi(x)^2}{H}.$$
This implies with \eqref{H} that, denoting $K_\epsilon := [-l_\epsilon, l_\epsilon]$, for any $k \in \Z_+$, $f$ such that $|f| \leq \psi$ and $x \in \R$,
$$\delta_x P_{s+k \gamma, t+k\gamma}[f \1_{K_\epsilon^c}] \leq \epsilon~~\text{ and }~~\delta_x Q_{s, t}[f \1_{K_\epsilon^c}] \leq \epsilon.$$
Then, for any $k \in \Z_+$ and $f$ such that $|f| \leq \psi$, 
\begin{align}
    | \delta_x  P_{s+k \gamma, t+k\gamma}f -  \delta_x Q_{s, t}f| &\leq 2 \epsilon +  | \delta_x  P_{s+k \gamma, t+k\gamma}[f \1_{K_\epsilon}] -  \delta_x Q_{s, t}[f \1_{K_\epsilon}]| \\&\leq 2 \epsilon + (1+l_\epsilon^2) \| \delta_x P_{s+k\gamma, t+k\gamma} - \delta_x Q_{s, t} \|_{TV}
\end{align}
Hence, \eqref{convergence-tv} implies that, for $k$ large enough, for any $f$ bounded by $\psi$, 
 \begin{align}
    | \delta_x  P_{s+k \gamma, t+k\gamma}f -  \delta_x Q_{s, t}f| &\leq 3 \epsilon,
\end{align}
implying that 
$$\| \delta_x  P_{s+k \gamma, t+k\gamma} -  \delta_x Q_{s, t} \|_{\psi} \underset{k \to \infty}{\longrightarrow} 0.$$
\end{proof}
We now prove Lemma \ref{gaussian}. 
\begin{proof}[Proof of Lemma \ref{gaussian}]
Defining $$f_\nu(x) := e^{-\frac{(x-a)^2}{2 {b_-}^2}} \land e^{-\frac{(x+a)^2}{2 {b_-}^2}},$$
we conclude easily that, for any $m \in [-a,a]$ and $\sigma \geq b_-$, for any $x \in \R$,
$$e^{-\frac{(x-m)^2}{2 \sigma^2}} \geq f_\nu(x).$$
Imposing moreover that $\sigma \leq b_+$, one has
$$\frac{1}{\sqrt{2 \pi} \sigma} e^{-\frac{(x-m)^2}{2 \sigma^2}} \geq \frac{1}{\sqrt{2 \pi} b_+} f_\nu(x),$$
which concludes the proof. 
\end{proof}

\subsection{Quasi-ergodic distribution for Brownian motion absorbed by an asymptotically periodic moving boundary}

Let  $(W_t)_{t \geq 0}$ be a one-dimensional Brownian motion and $h$  be a $\cC^1$-function such that $$h_{\min} := \inf_{t \geq 0} h(t) > 0,~~~~\text{ and }~~~~h_{\max} := \sup_{t \geq 0} h(t) < + \infty.$$
We assume also that 
$$-\infty < \inf_{t \geq 0} h'(t) \leq \sup_{t \geq 0} h'(t) < + \infty.$$
Denote by 
$$\tau_h := \inf\{t \geq 0 : |W_t| \geq h(t)\}.$$
Since $h$ is continuous, the hitting time $\tau_h$ is a stopping time with respect to the natural filtration of $(W_t)_{t \geq 0}$. Moreover, since $\sup_{t \geq 0} h(t) < + \infty$ and $\inf_{t \geq 0} h(t) > 0$,
$$\P_{s,x}[\tau_h < + \infty] = 1 ~~\text{ and }~~ \P_{s,x}[\tau_h > t] > 0,~~~~\forall s \leq t, \forall x \in [-h(s),h(s)].$$
The main assumption on the function $h$ is the existence of a $\gamma$-periodic function $g$ such that $h(t) \leq g(t)$, for any $t \geq 0$, and such that
$$h \sim_{t \to \infty} g,~~\text{and}~~h' \sim_{t \to \infty} g'.$$
Similarly to $\tau_h$, denote 
$$\tau_g := \inf\{t \geq 0 : |W_t| = g(t)\}.$$
Finally, let us assume that there exists $n_0 \in \N$ such that, for any $s \geq 0$, 
\begin{equation}
    \label{argmin}
    \inf \{ u \geq s : h(u) = \inf_{t \geq s} h(t)\} - s \leq n_0 \gamma.
\end{equation}
This condition says that there exists $n_0 \in \N$ such that, for any time $s \geq 0$, the infimum of the function $h$ on the domain $[s, + \infty)$ is reached on the subset $[s, s + n_0 \gamma]$. \medskip

We first show the following proposition. 
\begin{proposition}
The Markov process $(W_{t})_{t \geq 0}$, considered as absorbed by $h$ or by $g$, satisfies Assumption (A').  
\end{proposition}

\begin{proof}
In what follows, Assumption (A') w.r.t. the absorbing function $h$ will be shown. The following proof could easily be adapted for the function $g$.
\begin{itemize}
\item \textit{Proof of (A'1).}
 Denote $\cT := \{s \geq 0 : h(s) = \inf_{t \geq s} h(t)\}$. The condition \eqref{argmin} implies that this set contains an infinity of times.
 
 In what follows, the following notation is needed: for any $z \in \R$, define $\tau_z$ as
 $$\tau_z := \inf\{t \geq 0 : |W_t| = z\}.$$
 Also, let us state that, since the Brownian motion absorbed at $\{-1,1\}$ satisfies Assumption (A) of \cite{CV2014} at any time (see \cite{CV2017b}), it follows that, for a given $t_0 > 0$, there exists $c > 0$ and $\nu \in \cM_1((-1,1))$ such that, for any $x \in (-1,1)$,
\begin{equation}
    \label{min}
    \P_{0,x}\left[W_{\frac{t_0}{h_{\max}^2} \land t_0} \in \cdot \middle| \tau_1 > \frac{t_0}{h_{\max}^2} \land t_0\right] \geq c \nu.
\end{equation}
Moreover, regarding the proof of \cite[Section 5.1]{CV2017b}, the probability measure $\nu$ can be expressed as
\begin{equation}
    \label{def-nu}
    \nu = \frac{1}{2}\left(\P_{0,1-\epsilon}[W_{t_2} \in \cdot | \tau_1 > t_2] + \P_{0,-1+\epsilon}[W_{t_2} \in \cdot | \tau_1 > t_2]\right),
\end{equation}
 for some $0 < t_2 < \frac{t_0}{h_{\max}^2} \land t_0$ and $\epsilon \in (0,1)$. 
 
 The following lemma is very important for the following.
 \begin{lemma}
 \label{lemma-2}
 For all $z \in [h_{\min},h_{max}]$,
$$\P_{0,x}[W_{u} \in \cdot | \tau_z > u] \geq c \nu_z,~~~~\forall x \in (-z,z), \forall u \geq t_0,$$
where $t_0$ was evoked before, $c > 0$ is the same constant as in \eqref{min} and  
$$\nu_z(f) = \int_{(-1,1)} f(z x) \nu(dx),$$
with $\nu \in \cM_1((-1,1))$ defined in \eqref{def-nu}. 
 \end{lemma}
 The proof of this lemma is postponed after the current proof.
 
 Let $s \in \cT$. Then, for all $x \in (-h(s),h(s))$ and $t \geq 0$,   
$$\P_{s,x}[W_{s+t} \in \cdot | \tau_{h} > s+t] \geq \frac{ \P_{s,x}[\tau_{h(s)} > s+t]}{\P_{s,x}[\tau_h > s+t]} \P_{s,x}[W_{s+t} \in \cdot | \tau_{h(s)} > s+t],$$
By Lemma \ref{lemma-2}, for all $x \in (-h(s),h(s))$ and $t \geq t_0$,
$$\P_{s,x}[W_{s+t} \in \cdot | \tau_{h(s)} > s+t] \geq c \nu_{h(s)},$$
which implies therefore that, for any $t \in [t_0, t_0 + n_0 \gamma]$,
\begin{align}
    \P_{s,x}[W_{s+t} \in \cdot | \tau_h > s+t] &\geq \frac{ \P_{s,x}[\tau_{h(s)} > s+t]}{\P_{s,x}[\tau_h > s+t]} c \nu_{h(s)} \notag \\
    &\geq \frac{ \P_{s,x}[\tau_{h(s)} > s+t_0+n_0 \gamma]}{\P_{s,x}[\tau_h > s+t_0]} c \nu_{h(s)}.
    \label{change}
\end{align}
Let us introduce the process $X^h$ defined by, for all $t \geq 0$, $$X^h_t := \frac{W_t}{h(t)}.$$
By Itô's formula, for any $t \geq 0$,
$$X^h_t = X^h_0 + \int_0^t \frac{dW_s}{h(s)} - \int_0^t \frac{h'(s)}{h(s)} X^h_s ds.$$
Denote by $(M^h_t)_{t \geq 0} := \left(\int_0^t \frac{1}{h(s)}dW_s\right)_{t \geq 0}$. By the Dubins-Schwarz theorem, it is well-known that the process $M^h$ has the same law as $\left(W_{\int_0^t \frac{1}{h^2(s)}ds}\right)_{t \geq 0}$. Then, denoting $I^h(s) := \int_0^s \frac{1}{h^2(u)}du$ and, for any $s \leq t$ and for any trajectory $w$,
\begin{align}
    \cE_{s,t}^h(w) := &\sqrt{\frac{h(t)}{h(s)}} \exp\left(-\frac{1}{2} \left[h'(t)h(t) w_{I^h(t)}^2 - h'(s)h(s) w_{I^h(s)}^2 + \int_{s}^{t} w_{I^h(u)}^2[(h'(u))^2 - [h(u)h'(u)]']du\right]\right),
    \label{girsanov-h}
\end{align}
Girsanov theorem implies that, for all $x \in (-h(s),h(s))$, 

    \begin{equation}\P_{s,x}[\tau_h > s+t_0] = \E_{I^h(s),\frac{x}{h(s)}}\left[\cE_{s,s+t_0}^h(W) \1_{\tau_1 > \int_0^{s+t_0} \frac{1}{h^2(u)} du}\right].
\end{equation}
On the event $\{\tau_1 > \int_0^{s+t_0} \frac{1}{h^2(u)} du\}$, and since $h$ and $h'$ are bounded on $\R_+$, the random variable $\cE_{s,s+t_0}^h(W)$ is almost surely bounded by a constant $C > 0$, uniformly in $s$, such that for all $x \in (-h(s),h(s))$,
\begin{equation}\E_{I^h(s),\frac{x}{h(s)}}\left[\cE_{s,s+t_0}^h(W) \1_{\tau_1 > \int_0^{s+t_0} \frac{1}{h^2(u)} du}\right] \leq C \P_{0,\frac{x}{h(s)}}\left[\tau_1 > \int_s^{s+t_0} \frac{1}{h^2(u)} du\right].\end{equation}
Since $h(t) \geq h(s)$ for all $t \geq s$ (since $s \in \cT$), $I^h(s+t_0) - I^h(s) \leq \frac{t_0}{h(s)^2}$. By the scaling property of the Brownian motion and by Markov property, one has for all $x \in (-h(s),h(s))$, 
\begin{align*}\P_{s,x}[\tau_{h(s)} > s+t_0] &= \P_{0,x}[\tau_{h(s)} > t_0] \\ &= \P_{0,\frac{x}{h(s)}}\left[\tau_1 > \frac{t_0}{h^2(s)}\right] \\&= \E_{0,\frac{x}{h(s)}}\left[\1_{\tau_1 > \int_s^{s+t_0} \frac{1}{h^2(u)}du} \P_{0,W_{ \int_s^{s+t_0} \frac{1}{h^2(u)}du}}\left[\tau_1 > \frac{t_0}{h^2(s)} - \int_s^{s+t_0} \frac{1}{h^2(s)}ds\right]\right] \\ &= \P_{0,\frac{x}{h(s)}}\left[\tau_1 > \int_s^{s+t_0} \frac{1}{h^2(u)} du\right] \P_{0,\phi_{I^h(s+t_0) - I^h(s)}(\delta_x)}\left[\tau_1 > \frac{t_0}{h^2(s)} - \int_s^{s+t_0} \frac{1}{h^2(u)}du\right],\end{align*}
where, for any initial distribution $\mu$ and any $t \geq 0$, 
$$\phi_{t}(\mu) := \P_{0,\mu}[W_t \in \cdot | \tau_1 > t].$$

The family $(\phi_t)_{t \geq 0}$ satisfies the equality $\phi_t \circ \phi_s = \phi_{t+s}$ for all $s,t \geq 0$. By this property, and using that $I^h(s+t_0) - I^h(s) \geq \frac{t_0}{h_{\max}^2}$ for any $s \geq 0$, the minorization \eqref{min} implies that, for all $s \geq 0$ and $x \in (-1,1)$,
$$\phi_{I^h(s+t_0) - I^h(s)}(\delta_x) \geq c \nu.$$
Hence, by this minorization, and using that $h$ is upper-bounded and lower-bounded positively on $\R_+$, one has for all $x \in (-1,1)$,
$$\P_{0,\phi_{I^h(s+t_0)-I^h(s)}(\delta_x)}\left[\tau_1 >  \frac{t_0}{h^2(s)} - \int_s^{s+t_0} \frac{1}{h^2(u)}du\right] \geq c \P_{0,\nu}\left[\tau_1 >  \inf_{s \geq 0} \left\{\frac{t_0}{h^2(s)} - \int_s^{s+t_0} \frac{1}{h^2(u)}du\right\}\right],$$
that is to say, 
$$\frac{\P_{s,x}[\tau_{h(s)} > s+t_0]}{\P_{0,\frac{x}{h(s)}}\left[\tau_1 > \int_s^{s+t_0} \frac{1}{h^2(u)} du\right]} \geq c \P_{0,\nu}\left[\tau_1 >  \inf_{s \geq 0} \left\{\frac{\gamma}{h^2(s)} - \int_s^{s+t_0} \frac{1}{h^2(u)}du\right\}\right].$$
In other words, we just showed that, for all $x \in (-h(s),h(s))$, 
\begin{equation}
    \label{change2}
    \frac{ \P_{s,x}[\tau_{h(s)} > s+t_0]}{\P_{s,x}[\tau_h > s+t_0]} \geq \frac{c}{C} \P_{0,\nu}\left[\tau_1 >  \inf_{s \geq 0} \left\{\frac{t_0}{h^2(s)} - \int_s^{s+t_0} \frac{1}{h^2(u)}du\right\}\right] > 0.
\end{equation}
Moreover, by Lemma \ref{lemma-2} and the scaling property of the Brownian motion, for all $x \in (-h(s),h(s))$,
\begin{align}
\frac{\P_{s,x}[\tau_{h(s)} > s + t_0 + n_0 \gamma]}{\P_{s,x}[\tau_{h(s)} > s+ t_0]} &= \P_{0,\P_{0,x}[W_{t_0} \in \cdot | \tau_{h(s)} > t_0]}[\tau_{h(s)} > n_0 \gamma] \notag \\ &\geq c \P_{0,\nu_{h(s)}}[\tau_{h(s)} > n_0 \gamma] \notag \\ &= c \int_{(-1,1)} \nu(dy) \P_{h(s)y}[\tau_{h(s)} > n_0 \gamma] \notag
\\ &\geq c \P_{0,\nu}\left[\tau_1 > \frac{n_0 \gamma}{h_{\min}^2}\right] > 0. \label{change3}
\end{align}
Thus, gathering \eqref{change}, \eqref{change2} and \eqref{change3}, for any $x \in (-h(s),h(s))$ and any $t \in [t_0, t_0 + n_0 \gamma]$,
\begin{equation}\label{minorization}\P_{s,x}[W_{s+t} \in \cdot | \tau_h > s+t] \geq c_1 \nu_{h(s)},\end{equation}
where $$c_1 := c \P_{0,\nu}\left[\tau_1 > \frac{n_0 \gamma}{h_{\max}^2}\right] \times \frac{c}{C} \P_{0,\nu}\left[\tau_1 >  \inf_{s \geq 0} \left\{\frac{\gamma}{h^2(s)} - \int_s^{s+\gamma} \frac{1}{h^2(u)}du\right\}\right] c.$$ 

We recall that the Doeblin condition \eqref{minorization} is, for now, only obtained for $s \in \cT$. Consider now $s \not \in \cT$. Then, by the condition \eqref{argmin}, there exists $s_1 \in \cT$ such that $s < s_1 \leq s + n_0 \gamma$. Markov property and \eqref{minorization} implies therefore that, for any $x \in (-h(s),h(s))$, 
$$\P_{s,x}[W_{s+t_0 + n_0 \gamma} \in \cdot | \tau_h > s+t_0 + n_0 \gamma] = \P_{s_1,\phi_{s_1,s}}[W_{s+t_0 + n_0 \gamma} \in \cdot | \tau_h > s+t_0 + n_0\gamma] \geq c_1 \nu_{h(s_1)},$$
where, for all $s \leq t$ and $\mu \in \cM_1((-h(s),h(s)))$, $$\phi_{t,s}(\mu) := \P_{s,\mu}[W_t \in \cdot | \tau_h > t].$$ This concludes the proof of (A'1).
\item \textit{Proof of (A'2).} Since $(W_t)_{t \geq 0}$ is a Brownian motion, note that, for any $s \leq t$,
$$\sup_{x \in (-1,1)} \P_{s,x}[\tau_h > t] = \P_{s,0}[\tau_h > t].$$
Also, for any $a \in (0,h(s))$,
$$\inf_{[-a,a]}\P_{s,x}[\tau_h > t] = \P_{s,a}[\tau_h > t].$$
Thus, by Markov property, and using that the function $s \mapsto \P_{s,0}[\tau_g > t]$ is non-decreasing on $[0,t]$ (for all $t \geq 0$), one has for any $s \leq t$,
\begin{align}
\label{example}
    \P_{s,a}[\tau_h > t] \geq \E_{s,a}[\1_{\tau_0 < s+\gamma < \tau_h} \P_{\tau_0,0}[\tau_h > t]] \geq \P_{s,a}[\tau_0 < s+\gamma < \tau_h] \P_{s,0}[\tau_h > t].
\end{align}
Denoting $a := \frac{h_{\min}}{h_{\max}}$, by Lemma \ref{lemma-2} and taking $s_1 := \inf\{u \geq s : u \in \cT\}$, one obtains that, for all $s \leq t$,
 \begin{align*}\P_{s,\nu_{h(s_1)}}[\tau_h > t] &= \int_{(-1,1)} \nu(dx) \P_{s,h(s_1)x}[\tau_h > t] \\ &\geq \nu([-a,a]) \P_{s,h(s_1)a}[\tau_h > t] \\ & \geq \nu([-a,a]) \P_{0,h_{\min}}[\tau_0 < \gamma < \tau_h] \sup_{x \in (-h(s),h(s))} \P_{s,x}[\tau_h > t].\end{align*}
 This concludes the proof since, using \eqref{def-nu}, one has $\nu([-a,a]) > 0$.
\end{itemize}
\end{proof}
We now prove Lemma \ref{lemma-2}. 
\begin{proof}[Proof of Lemma \ref{lemma-2}]
This result comes from the scaling property of a Brownian motion. In fact, for any $z \in [h_{\min},h_{\max}]$, $x \in (-z,z)$ and $t \geq 0$, and for any measurable bounded function $f$,
\begin{align*}
    \E_{0,x}[f(W_t) | \tau_z > t] &= \E_{0,x}\left[f\left(z \times \frac{1}{z}W_{z^2 \frac{t}{z^2}}\right) \middle| \tau_z > t\right] \\
    &= \E_{0,\frac{x}{z}}\left[f\left(z \times W_{\frac{t}{z^2}}\right) \middle| \tau_1 > \frac{t}{z^2}\right].
\end{align*}
Then, the minorization \eqref{min} implies that, for any $x \in (-1,1)$,
$$\P_{0,x}\left[W_{\frac{t_0}{h_{\max}^2}} \in \cdot \middle| \tau_1 > \frac{t_0}{h_{\max}^2}\right] \geq c \nu.$$
This inequality holds for any time greater than $\frac{t_0}{h_{\max}^2}$. In particular, for any $z \in [h_{\min},h_{\max}]$ and $x \in (-1,1)$,
$$\P_{0,x}\left[W_{\frac{t_0}{z^2}} \in \cdot \middle| \tau_1 > \frac{t_0}{z^2}\right] \geq c \nu.$$
Then, for any $z \in [a,b]$, $f$ positive and measurable, and $x \in (-z,z)$, 
$$\E_{0,x}[f(W_{t_0}) | \tau_z > t_0] \geq c \nu_z\left(f\right),$$
where $\nu_z(f) := \int_E f(z \times x)\nu(dx)$.
This completes the proof of Lemma \ref{lemma-2}.

\end{proof}
The section is now concluded by stating and proving the following result.
\begin{theorem}
For any $s \leq t$, $n \in \N$ and any $x \in \R$, $$\P_{s+k \gamma, x}[\tau_h \leq t+k \gamma < \tau_g] \underset{k \to \infty}{\longrightarrow} 0.$$
In particular, Corollary \ref{qsd-cor} holds for $(W_t)_{t \geq 0}$ absorbed by $h$. 
\end{theorem}
\begin{proof} Reminding \eqref{girsanov-h}, by Markov property for the Brownian motion, one has for any $k,n \in \N$ and any $x \in \R$,
\begin{align*}
    \P_{s+k\gamma,x}[\tau_h > t+k\gamma] &= \sqrt{\frac{h(t+k\gamma)}{h(s+k \gamma)}} \E_{0,x}\left[\exp\left(-\frac{1}{2} \cA^h_{s,t,k}(W) \right) \1_{\tau_1 > I^h(t+k\gamma) - I^h(s+k \gamma)}\right], 
\end{align*}
where, for any trajectory $w = (w_u)_{u \geq 0}$,
\begin{multline*}
    \cA_{s,t,k}^h(w) =h'(t+k\gamma)h(t+k\gamma) w_{I^h(t+k\gamma) - I^h(s+k \gamma)}^2 - h'(s+k \gamma)h(s+k \gamma) w_0^2 \\ + \int_{0}^{t-s} w_{I^h(u+s+k\gamma) - I^h(s+k \gamma)}^2[(h'(u+s+k\gamma))^2 - [h(u+s+k \gamma)h'(u+s+k\gamma)]']du.
\end{multline*}
Since $h \sim_{t \to \infty} g$, one has for any $s,t \in [0, \gamma]$
$$\sqrt{\frac{h(t+k\gamma)}{h(s+k \gamma)}} \underset{k \to \infty}{\longrightarrow} \sqrt{\frac{g(t)}{g(s)}}.$$
For the same reasons, and using that the function $h$ is bounded on $[s+k \gamma, t+k\gamma]$ for all $s \leq t$, Lebesgue's dominated convergence theorem implies that 
$$I^h(t+k\gamma) - I^h(s+k\gamma) \underset{k \to \infty}{\longrightarrow} I^g(t) - I^g(s)$$
for all $s \leq t \in [0,\gamma]$. 
Moreover, since $h \sim_{t \to \infty} g$ and $h' \sim_{t \to \infty} g'$, one has for all trajectories $w = (w_u)_{u \geq 0}$ and $s \leq t \in [0,\gamma]$,  
$$ \cA_{s,t,k}^h(w) \underset{k \to \infty}{\longrightarrow} g'(t)g(t) w_{I^g(t) - I^g(s)}^2 - g'(s)g(s)w_0^2 + \int_{s}^{t} w_{I^g(u)}^2[(g'(u))^2 - [g(u)g'(u)]']du.$$
Since the random variable $\exp\left(-\frac{1}{2} \cA^h_{s,t,k}(W) \right) \1_{\tau_1 > I^h(t+k\gamma) - I^h(s+k \gamma)}$ is bounded almost surely, Lebesgue's dominated convergence theorem implies that 
$$\P_{s+k \gamma, x}[\tau_h > t+k \gamma] \underset{k \to \infty}{\longrightarrow} \P_{s,x}[\tau_g > t],$$
which concludes the proof. 
\end{proof}



\section*{Acknowledgement}

I would like to thank the anonymous reviewers for their valuable and relevant comments and suggestions, as well as Oliver Kelsey Tough for reviewing a part of this paper.

\bibliographystyle{abbrv}
\bibliography{biblio}

\end{document}